\documentclass[a4paper,reqno,12pt]{amsart}
\usepackage[T1,T2A]{fontenc}
\usepackage[english]{babel}
\usepackage{amsmath, amssymb, amsfonts, amsthm, amscd, mathrsfs,stmaryrd}
\usepackage{enumitem}
\usepackage[cmtip,arrow]{xy}
\usepackage{pb-diagram, pb-xy}
\usepackage[hidelinks]{hyperref}
\usepackage{longtable}
\setlist[enumerate]{label = (\alph*), ref=(\text{\alph*)}}
\setlist[itemize]{nolistsep}
\usepackage{epsfig}
\usepackage{graphicx}
\def\cN{{\rm N}}
\def\GG{{\mathbb G}}%
\def\FF{{\mathbb F}}%
\def\CC{{\mathbb C}}%
\def\KK{{\mathbb K}}%
\def\ZZ{{\mathbb Z}}%
\def\RR{{\mathbb R}}%
\def\NN{{\mathbb N}}%
\def\QQ{{\mathbb Q}}%
\def\PP{{\mathbb P}}%
\def\AA{{\mathbb A}}%

\def\RRR{\mathcal{R}}

\def\Of{{\mathcal{O}}}
\def\OH{{\mathcal{H}}}
\def\Mat{{\rm Mat}}%
\def\LND{{\rm LND}}%
\def\Cl{{\rm Cl}}%
\def\Aut{{\rm Aut}}%
\def\rk{{\rm rk}}%
\def\reg{{\rm reg}}%
\def\SAut{{\rm SAut}}%
\def\Aff{{\rm Aff}}%
\def\Tame{{\rm Tame}}%
\def\GL{{\rm GL}}%
\def\PGL{{\rm PGL}}%
\def\SL{{\rm SL}}%
\def\SO{{\rm SO}}%
\def\Sp{{\rm Sp}}%
\def\sl{\mathfrak{sl}}%

\def\Div{{\rm Div}}%
\def\WDiv{{\rm WDiv}}
\def\exp{{\rm exp}}%
\def\deg{{\rm deg}}%
\def\codim{{\rm codim}}%
\def\trdeg{{\rm tr.deg}}%
\def\Susp{{\rm Susp}}%
\def\Spec{{\rm Spec\,}}%
\def\Id{{\rm Id}}%
\newtheorem{theorem}{Theorem}
\newtheorem{proposition}{Proposition}
\newtheorem{lemma}{Lemma}
\newtheorem{corollary}{Corollary}
\newtheorem{conjecture}{Conjecture}
\theoremstyle{definition}
\newtheorem{definition}{Definition}

\theoremstyle{remark}
\newtheorem{remark}{Remark}
\newtheorem{example}{Example}

\newcounter{property}

\newcounter{prooperty}

\makeatletter
\begin{document}
\title[Automorphisms and infinite transitivity]{Automorphisms  of algebraic varieties and \\ infinite transitivity}
\author{Ivan Arzhantsev}
\address{Faculty of Computer Science, HSE University, Pokrovsky Boulevard 11, Moscow, 109028 Russia}
\email{arjantsev@hse.ru}
\thanks{Supported by Russian Foundation for Basic Research, grant 20-11-50106}
\subjclass[2010]{Primary 14L30, 14R10; \ Secondary 13E10, 14M25, 20M32}
\keywords{Algebraic variety, automorphism, group action, multiple transitivity, locally nilpotent derivation, toric variety, affine cone} 

\begin{abstract}
We survey recent results on multiple transitivity of automorphism groups of affine algebraic varieties. We consider the property of infinite transitivity of the special automorphism group,
which is equivalent to flexibility of the corresponding affine variety. These properties have important algebraic and geometric consequences. At the same time they are fulfilled for wide classes of varieties. Also we study situations where infinite transitivity takes place for automorphism groups generated by finitely many one-parameter subgroups. In the appendices to the paper, the results on infinitely transitive actions in complex analysis and in combinatorial group theory are discussed.
\end{abstract}
\maketitle
\tableofcontents
\section*{Introduction}
\label{int}

In this survey we study multiply transitive actions of automorphism groups of algebraic varieties and the closely related concept
of a flexible affine algebraic variety. Let us start by reminding basic definitions.

Let us fix a positive integer $m$. We say that an action of a group $G$ on a set $X$ is \emph{$m$-transitive} if for any two tuples of $m$ pairwise distinct points of the set $X$ there exists an element of the group $G$ that sends the first tuple to the second one. The action of $G$ on $X$ is called \emph{infinitely transitive} if it is $m$-transitive for any positive integer $m$. Infinite transitivity is equivalent to the fact that the pointwise stabilizer of any finite set of points acts transitively on the complement to this set.

In the case of finite groups, it is easy to see that the symmetric group $S_n$ acts on a set with $n$ elements $n$-transitively. The same action restricted to the subgroup of even permutations $A_n$ becomes $(n-2)$-transitive. The surprising fact is that all other finite permutation groups have the transitivity degree not greater than~$5$~\cite[p.~2.1]{DM}. Moreover, only the Mathieu groups $M_{12}$ and $M_{24}$ have the transitivity degree~$5$, and the transitivity degree equals $4$ only for the Mathieu groups $M_{11}$ and~$M_{23}$~\cite[Chapter~6]{DM}. At the same time, there are infinitely many finite groups of the transitivity degree~$3$. For example, we may consider the action of the group $\PGL_2(\FF_q)$ on the projective line $\PP^1(\FF_q)$ over a field with $q$ elements. With different values of $q$ we obtain different $3$-transitive actions.

Let us move to the case of infinite groups. First of all, we are interested in automorphism groups of algebraic varieties. It is well known that the connected component of the automorphism group of a complete (in particular, projective) algebraic variety is a (finite-dimensional) algebraic group~\cite{Mat}. In particular, such a group can not act on a variety infinitely transitively.

More precisely, it is known that the maximum of transitivity degree of the action of an algebraic group is $3$ and this value is achieved only for the action of the group $\PGL_2$ on the projective line $\PP^1$~\cite[Corollary~2]{Kn}. All $2$-transitive actions of algebraic groups over an arbitrary algebraically closed field are classified in~\cite{Kn}. It turned out that in addition to the action of the group $\PGL_{n+1}$ on $\PP^n$, such actions are only the actions of the semidirect product $G\rightthreetimes V$ on~$V$, where $G$ is a linear algebraic group and $V$ is a rational $G$-module such that $G$ acts on $V\setminus\{0\}$ transitively. All such actions can be explicitly listed. Multiply transitive actions of Lie groups were studied in~\cite{Bo, Kr}.

In the case of affine varieties, the situation changes essentially. It is not difficult to show that for $n\ge 2$ the automorphism group of the affine space $\AA^n$ over an infinite field acts on $\AA^n$ infinitely transitively. Constructing other examples of infinitely transitive actions on algebraic varieties is a nontrivial task that requires to develop a special technique. One of the first works in this direction is the paper~\cite{KZ}. Here the infinite transitivity is proved for the action of the automorphism group on the smooth locus of certain affine hypersurfaces. Such hypersurfaces are called suspensions. In~\cite{AKZ-1}, the concept of a flexible variety is defined, and the properties of flexibility and infinite transitivity were
proved for affine cones over varieties of generalized flags, non-degenerate affine toric varieties and suspensions over flexible varieties.

One interpretation of the infinite transitivity of the action of the automorphism group is that any two embeddings of a given finite set into the smooth locus of a flexible variety are equivalent, that is, they can be translated to each other by a suitable automorphism of the variety. This statement is consonant with the Abyankar-Moh-Suzuki Theorem on the equivalence of embeddings of a line into a plane.

Unless otherwise stated, we assume the ground field ~$\KK$ to be an algebraically closed field of characteristic zero. Consider the additive group $\GG_a$ of the field $\KK$ with the natural structure of a one-dimensional linear algebraic group. For each non-identical regular action $\GG_a\times X\to X$, the image of $\GG_a$ in the automorphism group $\Aut(X)$ is called a \emph{$\GG_a$-subgroup}. Denote by $\SAut(X)$ the subgroup of the group $\Aut(X)$ generated by all $\GG_a$-subgroups. Since a subgroup conjugated to a $\GG_a$-subgroup by an arbitrary automorphism is again a $\GG_a$-subgroup, the subgroup $\SAut(X)$ is normal in~$\Aut(X)$.

We say that a smooth point $x$ of a variety $X$ is \emph{flexible} if the tangent space $T_x(X)$ is generated by tangent vectors to the orbits of $\GG_a$-actions on $X$ passing through~$x$. An algebraic variety $X$ is called \emph{flexible} if every its smooth point is flexible. For example, this condition is satisfied if there is at least one flexible point on $X$ and the group $\Aut(X)$ acts transitively on the smooth locus.

We are ready to formulate the main result of the paper~\cite{AFKKZ-1}. This is one of the central results for this survey.

\begin{theorem} \label{tmain}
Let $X$ be an irreducible affine algebraic variety of dimension $\ge 2$. Then the following conditions are equivalent.
\begin{enumerate}
\item[(i)]
The group $\SAut(X)$ acts on the smooth locus $X^{\reg}$ of the variety $X$ transitively.
\item[(ii)]
The group $\SAut(X)$ acts on the smooth locus $X^{\reg}$ infinitely transitively.
\item[(iii)]
The variety $X$ is flexible.
\end{enumerate}
\end{theorem} 

In other words, the transitivity of the group $\SAut(X)$ implies infinite transitivity, and these conditions are equivalent to the flexibility condition formulated in infinitesimal terms. Note that in \cite[Theorem~2]{APS} and \cite[Theorem~1.11]{FKZ} Theorem~\ref{tmain} is generalized to the case of irreducible quasi-affine varieties of dimension~$\ge 2$.

Let us take a closer look at the structure of this paper. In the first section, following~\cite{AFKKZ-1}, we define algebraically generated subgroups of the automorphism group of an algebraic variety as subgroups generated by a family of algebraic subgroups. Note that the group $\SAut(X)$ defined above is algebraically generated. Actions of such subgroups have a number of properties in common with the actions of finite-dimensional groups. Such properties include the local closedness of orbits, as well as an analogue of Rosenlicht's theorem on the separation of orbits by rational invariants and an analogue of Kleiman's theorem on transversality. After that we proceed to the discussion on locally nilpotent derivations of affine algebras. The technique of locally nilpotent derivations and the construction of replicas of $\GG_a$-subgroups are among the main technical tools of this theory. The discussion of this technique ends with the proof of  equivalence of conditions (i)\! and (iii)\! in Theorem~\ref{tmain}. Finally, in the last part of this section we discuss results which lead to implication (i)~$\!\Rightarrow$\!~(ii) in Theorem~\ref{tmain}. In fact, we deal with more general setup: infinite transitivity is considered on an open orbit, which need not coincide with the smooth locus of the variety, and the acting group need not be $\SAut(X)$, but any subgroup of this group generated by a saturated family of $\GG_a$-subgroups.

\smallskip

Section~2 is devoted to the study of properties of flexible varieties established in \cite{AFKKZ-1,FKZ} and subsequent papers. We show that the existence of a flexible point on an affine variety is equivalent to the triviality of the Makar-Limanov field invariant and prove that every such variety is unirational. Bogomolov's conjecture~\cite{BKK} is also discussed here. It offers a characterization of unirational varieties in terms of stably infinitely transitive birational models. We consider collectively infinitely transitive actions, the $\AA^1$-richness property, and a generalization of the Gromov-Winkelmann Theorem to arbitrary flexible varieties obtained in the work of Flenner, Kaliman and Zaidenberg~\cite{FKZ}. The theorem claims that flexibility is preserved when passing to a quasi-affine variety obtained from a flexible affine variety by removing a subvariety of codimension $\ge 2$. We also consider the question of when the automorphism group of an affine variety $X$ acts on the smooth locus $2$-transitively. In the case when some affine algebraic group of positive dimension acts non-identically on $X$, this condition is equivalent to the flexibility of the variety~\cite{Ar}.

\smallskip

In Section~3 we consider examples of flexible varieties. The presentation begins with the already mentioned construction of suspension. Then non-degenerate affine toric varieties are considered. To prove flexibility, the root $\GG_a$-subgroups and related Demazure roots are used. A generalization of this result is the theorem on flexibility of non-degenerate affine
horospherical varieties proved in recent papers by Gaifullin and Shafarevich~\cite{Sh,GS}. Flexibility of affine cones over projective varieties is studied, in particuar over del Pezzo surfaces~\cite{PW,Pe-2} and higher dimensional Fano varieties~\cite{MPS,PZ-2}. Here the proof of flexibility uses the construction of a cylindrical subset on a projective variety proposed by Kishimoto, Prokhorov and Zaidenberg, see e.g.~\cite{CPPZ}. We discuss the flexibility property for a universal torsor and for a total coordinate space over varieties covered by open charts  isomorphic to an affine space~\cite{APS}. These objects arise in the theory of Cox rings. The flexibility of Gizatullin surfaces and Calogero-Moser varieties is also considered.

\smallskip

In Section~4 we discuss a recently discovered effect --- infinite transitivity can occur for subgroups of the automorphism group generated by a finite number of $\GG_a$-subgroups. In the case of affine spaces, this result can be extracted from the theorems of Derksen and Bodnarchuk. In~\cite{AKZ-2}, we formulate a conjecture that for any flexible affine variety there exists such a finite set of one-dimensional unipotent subgroups that the group generated by them acts infinitely transitively on its open orbit. This conjecture is proved in~\cite{AKZ-2} for non-degenerate affine toric varieties of dimension $\ge 2$, which are nonsingular in codimension $2$. It is known that in the case of an affine space, three generating $\GG_a$-subgroups are sufficient for infinite transitivity. The ind-structure on the automorphism group $\Aut(X)$ and the passage from the subgroup $G$ to its closure $\overline{G}$ in the ind-topology play an important role here. In the last part, we follow the paper~\cite{AZ} and discuss the Tits alternative for groups generated by a finite number of root subgroups in the automorphism group of an affine toric variety.

\smallskip

The last part of the survey consists of two appendices. Appendix~A is devoted to holomorphic flexibility and infinite transitivity for groups of automorphisms of complex analytic manifolds, as well as to the explanation of the connection of concepts and constructions introduced above with the Andersen-Lempert theory, Gromov sprays, Oka manifolds and Oka maps.

In Appendix~B we prove several results on infinite transitivity of actions of abstract groups. To the reader who has not worked with infinitely transitive actions before, we recommend to read this appendix before reading the main part of the text. Here we prove the infinite transitivity for a normal subgroup of an infinitely transitive group. Following~\cite{MD}, we construct an example of an infinitely transitive action of a $2$-generated group and give a brief overview of the known facts on existence of infinitely transitive effective actions of finitely generated groups. Such actions are implemented both on abstract sets and on infinite graphs, metric spaces and other natural objects. Recently, this direction has been actively developing, but this happens without any connection with the results on infinite transitivity of automorphism groups of algebraic varieties. It is natural to expect that in the future the two theories will enrich each other.

\smallskip

Note also that the survey contains some original results and observations. For example, Proposition~\ref{relnew} seems to be new. 

\smallskip

In our opinion, the theory of flexible varieties is currently at its zenith. On the one hand, this area has proved its right to exist: a number of general structural theorems have been obtained, interesting applications and consequences of the results on flexibility have been found, it has been shown that wide classes of affine varieties have the flexibility property, nontrivial connections with other areas of mathematics have been established. On the other hand, a number of important open questions and conjectures are connected with this topic, which seem to be accessible, including within the framework of already developed techniques, and which are interesting to work on. We are confident that in the coming years the theory of flexible varieties will be supplemented with new results, and we hope that this survey will contribute to the further development of this area.

\smallskip

The author would like to take this opportunity to thank his co-authors and colleagues with whom he works together in this field or discusses related issues. He is sincerely grateful to Mikhail Zaidenberg, J\"urgen Hausen, Shulim Kaliman, Frank Kutzschebauch, Hubert Flenner, Hendrik S\"uss, Gene Freudenburg, Alvaro Liendo, Yuri Prokhorov, Dmitry Timashev, Karine Kuyumzhiyan, Sergey Gaifullin, Alexander Perepechko and many others. Special thanks are due to the referee for a careful reading of the text and many valuable comments.

\section{Basic facts on infinitely transitive actions and flexible varieties}
\label{sec-1}
    
\subsection{Infinite transitivity on an affine space} We start this section by proving a well-known fact: for $n\ge 2$, the group $\Aut(\AA^n)$ acts on $\AA^n$ infinitely transitively. This result hardly makes sense to associate with a specific work, it has a folklore nature and can be proved by a junior math student who is familiar with the concept of an interpolation polynomial. It is also an obvious consequence of Theorem~\ref{tmain}: the space $\AA^n$ is flexible, since the tangent space at each point is generated by tangent vectors to the orbits of one-parameter subgroups of parallel translations along the coordinate axes. Nevertheless, we give here a direct proof of the infinite transitivity of the action of the automorphism group, since it illustrates a number of ideas and methods underlying this theory. An affine space is considered here over an arbitrary infinite field.

\begin{proposition}
For $n\ge 2$, the group $\Aut(\AA^n)$ acts on $\AA^n$ infinitely transitively.
\end{proposition} 

\begin{proof}
We fix a positive integer $m$ and choose $m$ pairwise distinct points
$$
(x_{11},\ldots,x_{n1}),\ldots,(x_{1m},\ldots,x_{nm}).
$$
It suffices to prove that these points can be moved to a standard position by a suitable automorphism of the space $\AA^n$. Over a field of characteristic zero, we can assume that a standard position is the set
$$
(1,\ldots,1), \ldots, (m,\ldots,m),
$$
and over fields of positive characteristic, obvious changes need to be made in the arguments given below.

\smallskip

\noindent{\it Step $1$.}  We are going to ensure that $x_{1 i}\ne x_{1 j}$ for all $i\ne j$. We consistently increase the number of such inequalities. Let us suppose that $x_{11}=x_{12}$. Then we can assume that $x_{21}\ne x_{22}$. Let us find a polynomial $f(x)$ such that $f(x_{21})=a$ and $f(x_{2i})=0$ for all $x_{2i}\ne x_{21}$, where $a\ne x_{1i}-x_{1j}$ for all $i,j=1,\ldots,m$. Applying the automorphism $(x_1+f(x_2),x_2,\ldots,x_n)$ of $\AA^n$ to our set, we achieve that $x_{11}\ne x_{12}$ and the inequalities of the first coordinates of the points that took place before will not turn to equalities after such a transformation. Thus, after a finite number of such transformations, we obtain the required.

\smallskip

\noindent{\it Step $2$.}  Now we assume that $x_{1i}\ne x_{1j}$ for all $i\ne j$. Let us take such polynomials $f_2(x),\ldots,f_n(x)$ that $x_{si}+f_s(x_{1i})=i$ for all $s=2,\ldots,n$ and ${i=1,\ldots,m}$. After applying the automorphism 
$$
(x_1,x_2+f_2(x_1),\ldots,x_n+f_n(x_1))
$$
we obtain a set of points
$$
(x_{11},1,\ldots,1),\ldots,(x_{1m},m,\ldots,m). 
$$

\smallskip

\noindent{\it Step $3$.} Take a polynomial $h(x)$ such that
$$
x_{1i}+h(i)=i
$$
for all $i=1,\ldots,m$. Then the automorphism $(x_1+h(x_2),x_2,\ldots,x_n)$ moves our set to the standard position. 
\end{proof}

\begin{remark}
For $n=1$, the group $\Aut(\AA^1)$ is isomorphic to the semidirect product ${\GG_m\rightthreetimes\GG_a}$, where the torus $\GG_m$ acts by scalar multiplications and the group $\GG_a$ acts by parallel translations. This implies that $\Aut(\AA^1)$ acts on $\AA^1$ $2$-transitively, but not $3$-transitively.
\end{remark}
  
\subsection{Algebraically generated automorphism groups} Let $X$ be an algebraic variety over the ground field $\KK$. An \emph{algebraic subgroup} of the automorphism group $\Aut(X)$ is a subgroup $H\subseteq\Aut(X)$ with a structure of an algebraic group such that the action map $H\times X\to X$ is a morphism of algebraic varieties. We say that a subgroup $G$ of the group $\Aut(X)$ is \emph{algebraically generated} if it is generated as an abstract group by a family $S$ of connected algebraic subgroups of the group $\Aut(X)$.

\begin{proposition}[{\cite[Propositions~1.3 and~1.5]{AFKKZ-1}}] \label{ccff} 
Consider an algebraically generated subgroup $G\subseteq\Aut(X)$. Then
\begin{enumerate}
\item[(a)]
for any point $x\in X$, the orbit is $G.x$ is locally closed in $X$\textup;
\item[(b)]
there is such a finite set of not necessarily pairwise distinct subgroups $H_1,\ldots, H_s\in S$ that
$$
G.x=H_1.(H_2. \cdots (H_s.x))
$$
for any point $x\in X$. 
\end{enumerate}
\end{proposition} 

\begin{proposition}[{\cite[Proposition~1.8]{AFKKZ-1}}]
 \label{p18}
Suppose \textup that a generating family $S$ of connected algebraic subgroups is invariant under conjugation by elements of the group $G$. Then there is such a sequence $(H_1,\ldots, H_s)$ of subgroups in $S$ that for all points $x\in X$ the tangent space $T_x(G.x)$ to the orbit $G.x$ is generated by tangent spaces
$$
T_x(H_1.x),\ldots ,T_x(H_s.x).
$$
\end{proposition}

We consider an algebraically generated subgroup $G\subseteq\Aut(X)$ with a generating family $S$ of connected algebraic subgroups that is invariant under conjugation by elements of the group $G$. We say that a point $x\in X^{\reg}$ is \emph{$G$-flexible} if the tangent space $T_x X$ is generated by subspaces $T_x(H.x)$, where $H\in S$. Then~\cite[Corollary~1.11]{AFKKZ-1} states that a point $x\in X^{\reg}$ is $G$-flexible if and only if the orbit $G.x$ is open in $X$. Further, if such an open orbit exists, then it is unique and consists exactly of all $G$-flexible points in $X^\reg$. Note also that the definition of a $G$-flexible point does not depend on the choice of a generating family $S$ of the group $G$.

\smallskip

Using the semi-continuity theorem for the dimension of fibers of a morphism, it is not difficult to prove~\cite[Corollary~1.12]{AFKKZ-1} that for an algebraically generated group $G\subseteq\SAut(X)$, the function $X\to\ZZ_{\ge 0}$, $x\mapsto\dim G.x$ is lower semi-continuous. In particular, the union of orbits of maximum dimension in $X$ is a Zariski-open subset.

The following result is an analogue of the well-known Rosenlicht theorem for algebraic groups~\cite[Theorem~2.3]{VP-1}.

\begin{theorem}[{\cite[Theorem~1.13]{AFKKZ-1}}] 
Let $G\subseteq\Aut(X)$ be an algebraically generated subgroup. Then there is a finite set of rational $G$-invariant functions which separate $G$-orbits in general position on~$X$.
\end{theorem}

The proof of this result repeats almost verbatim the proof of the original Rosenlicht theorem given in~\cite{VP-1}.

\begin{corollary}[{\cite[Corollary~1.14]{AFKKZ-1}}]
 \label{ccdd}
Let an algebraically generated group $G$ act on an irreducible variety $X$. Then 
$$
\trdeg(\KK(X)^G:\KK)=\min_{x\in X}\,\{ {\codim}_{X} G.x\}\, .
$$
In particular, there is an open $G$-orbit on $X$ if and only if $\KK(X)^G=\KK$.
\end{corollary}

If an algebraic group $G$ acts transitively on an algebraic variety $X$ and $Y,Z$ are smooth subvarieties in~$X$, Kleiman's Transversality theorem claims that a typical translate $g.Z$ ($g\in G$) intersects $Y$ transversally. The following generalization of this result takes place in the case of algebraically generated subgroups.

\begin{theorem}[{\cite[Theorem~1.15]{AFKKZ-1}}] 
\label{ccgg}
Let an algebraically generated subgroup $G\subseteq\Aut(X)$ be generated by a family $S$ of connected algebraic subgroups that is invariant under conjugation by elements of $G$. Suppose that $G$ acts with the open orbit $O\subseteq X$. Then there are such subgroups $H_1,\ldots, H_s\in S$ that for any locally closed reduced subschemes $Y$ and $Z$ in $O$ one can find such a Zariski-open subset $U=U(Y,Z)\subseteq H_1\times\ldots\times H_s,$ that any elements $(h_1,\ldots, h_s)\in U$ satisfy the following conditions:
\begin{enumerate}
\item[(a)] the translate $(h_1\cdot\ldots\cdot h_s).Z^\reg$ intersects $Y^\reg$ transversely\textup;
\item[(b)] $\dim (Y\cap (h_1\cdot\ldots\cdot h_s).Z)\le
\dim Y+\dim Z-\dim X$.
\end{enumerate}
In particular\textup, if $\dim Y+\dim Z<\dim X,$ then $Y\cap (h_1\cdot\ldots\cdot h_s).Z=\varnothing$.
\end{theorem}

Note that in the work~\cite{Po} results on algebraically generated subgroups have been further generalized. A number of properties of actions including local closureness of orbits and an analogue of Rosenlicht's theorem on rational invariants, are proved here for any connected subgroup of the automorphism group of an arbitrary irreducible algebraic variety.

%%%%%%%%%%%%%%%%%%%%%%%%%%%%%%%%%%%%%%%

\subsection{Locally nilpotent derivations and replicas of $\GG_a$-subgroups}\label{1.3} For our further purposes it is important to know how to work with infinitesimal generators of $\GG_a$-subgroups of the automorphism group $\Aut(X)$ of an affine variety $X$. Below we formulate the related results.

\smallskip

(1) If the group $\GG_a$ acts on an affine variety $X=\Spec A$, then the derivation $\partial$ of the algebra $A$ associated with this action is \emph{locally nilpotent}, that is, for every $f\in A$ there is such an $n\in\NN$ that $\partial^n(f)=0$. This result is explained by the fact that each function $f$ lies in a finite-dimensional $\GG_a$-invariant subspace in $A$, and, by the Lie-Kolchin Theorem, in a suitable basis of this subspace the group $\GG_a$ acts by unitriangular matrices. Hence, the differential of such an action is given in this subspace by a nilpotent operator.

\smallskip

(2) Conversely, given a locally nilpotent $\KK$-linear derivation $\partial:A\to A$ and $t\in\KK$, the mapping $\exp(t\partial):A\to A$ defines an automorphism of the algebra $A$ and, consequently, of the variety~$X$. Thus, each locally nilpotent derivation $\partial\neq 0$ is associated with the $\GG_a$-subgroup $H=\exp(\KK\partial)$ in $\Aut(X)$; for more details see~\cite{Fr}. If we consider $\partial$ as a vector field on $X$, the action of $H$ on $X$ is exactly the phase flow associated with this field.

\smallskip

(3) The algebra of invariants $\KK[X]^H=\ker\partial$ has the transcendence degree ${\dim X-1}$. Further, each $H$-invariant function $f\in\KK[X]^H$ defines a new $\GG_a$-subgroup $H_f:=\exp(\KK f\partial)$. We call such subgroups \emph{replicas} of the subgroup $H$. Bellow such subgroups play a significant role.

The subgroup $H_f$ acts in the same direction as the subgroup $H$, but with a different velocity along the orbits. In particular, those $H$-orbits, where the function $f$ vanishes, consist of $H_f$-fixed points. The mapping $f\mapsto f\partial$ embeds $\ker\partial$ as an abelian subalgebra into the Lie algebra of all regular vector fields on $X$.

For a subgroup $G\subseteq\SAut(X)$, we denote by $\LND(G)$ the set of locally nilpotent vector fields on $X$ that generate $\GG_a$-subgroups contained in $G$. This set is invariant under conjugation by elements of $G$ and is a cone, that is, $\KK\cdot\LND(G)\subseteq \LND(G)$. In the following, we consider subsets $\cN\subseteq\LND(G)$ of locally nilpotent vector fields for which the associated subgroups $\{\exp(\KK\partial)\}$, $\partial\in\cN$ generate the group~$G$. For brevity, we say that $\cN$ generates the group~$G$ and write $G=\langle\cN\rangle$.

The following result can be deduced from Proposition~\ref{p18}.

\begin{corollary} [{\cite[Corollary~1.21]{AFKKZ-1}}] \label{c121}
\!For a subgroup $G\!=\!\langle\cN\rangle\!$ of the group $\Aut(X),$ where $\cN\subseteq\LND(G)$ is invariant under conjugation in $G$, there are locally nilpotent vector fields $\partial_1,\ldots,\partial_s\in\cN$ that generate the tangent space $T_x(G.x)$ at each point $x\in X$.
\end{corollary}

For a point $x\in X$, let $\LND_x(G)\subseteq T_x X$ denote the cone of all tangent vectors $\partial(x)$, where $\partial$ runs through $\LND(G)$. According to Corollary~\ref{c121}, we have $T_x(G.x)=\langle\LND_x(G)\rangle$. Thus, a point $x\in X^\reg$ is $G$\nobreakdash-flexible if and only if the cone $\LND_x(G)$ generates the tangent space~$T_x X$.

Applying Corollary~\ref{c121} to the group $G=\SAut(X)$, we obtain equivalence (i)$\Leftrightarrow$(iii) in Theorem~\ref{tmain} from the Introduction.

\begin{corollary}\label{tr-fl} 
For an irreducible affine variety $X$, the action of the group $\SAut(X)$ on $X^\reg$ is transitive if and only if $X$ is flexible.
\end{corollary}

\begin{example} \label{affspa}
We illustrate the notion of a replica and a special automorphism group $\SAut(X)$ in the case of the affine space $X=\AA^n$.
The group $\SAut(\AA^n)$ contains $\GG_a$-subgroups of parallel translations in all directions. The infinitesimal generator of such a subgroup coincides with the partial derivative in this direction, which is considered as a locally nilpotent derivation of the polynomial ring in $n$ variables. For example, in the case of the parallel translation along the first coordinate axis $(x_1+t,x_2,\ldots,x_n)$, $t\in\KK$, the corresponding derivation $\partial$ is $\frac{\partial}{\partial x_1}$. Replicas of such a subgroup are translations of the form $(x_1+f,x_2,\ldots,x_n)$ for an arbitrary polynomial $f\in\KK[x_2,\ldots,x_n]$.

As one more example, we consider the locally nilpotent derivation $\partial=x_1\frac{\partial}{\partial x_2}+x_2\frac{\partial}{\partial x_3}$ of the polynomial ring $\KK[x_1,x_2,x_3]$ and the invariant function ${f=x_2^2-2x_1x_3\in\ker\partial}$. The corresponding replica~$H_f$ contains the well-known automorphism of Nagata
$$
H_f(1)=\exp(f\cdot\partial)\in \SAut (\AA^3).
$$

Note that any automorphism from $\SAut(\AA^n)$ preserves the volume form on ~$\AA^n$. So we have
$$
\SAut(\AA^n)\subseteq J_n,
$$ 
where $J_n$ denotes the subgroup in $\Aut(\AA^n)$ consisting of all automorphisms such that the determinant of the Jacobi matrix is~$1$.

Let us recall that the subgroup of \emph{tame} automorphisms $\Tame_n\subseteq\Aut (\AA^n)$
is the subgroup generated by all {\em elementary} automorphisms $\sigma\in \Aut(\AA^n)$ of the form
\begin{align*}
\sigma&=\sigma_{i,\lambda,f} \colon (x_1, \ldots, x_{i-1},x_i, x_{i+1},\ldots,x_n)
\\
&\longmapsto (x_1,\ldots, x_{i-1}, \lambda x_i+f,x_{i+1}\ldots, x_n),
\end{align*}
where $f\in \KK[x_1, \ldots, x_{i-1}, x_{i+1}\ldots,x_n]$ and $\lambda\in\GG_m$. An elementary automorphism~$\sigma$ with $\lambda=1$ is contained in the $\GG_a$-subgroup $\{\sigma_{i,1,tf}\}_{t\in\KK}$. Since for any $\lambda\in\GG_m$ the conjugated element $h_\lambda\circ\sigma\circ h_\lambda^{-1}$, where $h_\lambda(x)=\lambda x$, is again elementary, any tame automorphism $\beta$ with Jacobi matrix having determinant~$1$ can be written as a product of elementary automorphisms with Jacobi matrix having determinant~$1$. Thus ${\Tame_n\cap J_n\subseteq\SAut (\AA^n)\subseteq J_n.}$

For $n=3$ the first inclusion is strict. In fact, the Shestakov-Umirbaev Theorem states that the Nagata automorphism $H_f(1)$ is wild or, in other words, 
$$
H_f(1)\in \SAut (\AA^3) \setminus (\Tame_3\cap J_3).
$$

For $n=2$, the Jung--van der Kulk Theorem states that
$$
\Aut(\AA^2)=\Tame_2
$$
and so $\SAut (\AA^2)=J_2$. It is also known that the group $J_2$ is perfect and its derived subgroup coincides with the derived subgroup of the group $\Aut(\AA^2)$. The question of whether the equality $\SAut(\AA^n)=J_n$ holds for $n\ge 3$ remains open, see ~\cite[Question~15.5.4]{FK}. Moreover, it is not known whether the group $J_n$ coincides with the closure of the subgroup $\SAut(\AA^n)$ with respect to the ind-topology on the group $\Aut(\AA^n)$.
\end{example}

\subsection{Flexibility and infinite transitivity} Let $X$ be an affine algebraic variety and ${G\subseteq\SAut(X)}$ be a subgroup generated by a set $\cN$ of locally nilpotent vector fields satisfying the following conditions:
\begin{enumerate}
\item$\cN$ is invariant under the action of the group $G$ by conjugation;
\item$\cN$ is closed under taking replicas.
\end{enumerate} 

We call such a generating set {\em saturated}. Note that a more essential condition here is condition~(b), since starting from an arbitrary set of locally nilpotent vector fields $\cN_0$ generating $G$ and satisfying condition~(b) we can add all conjugates by elements of $G$ to elements of $\cN_0$ and obtain a saturated set $\cN$ generating the same group $G$.

The following theorem proves (i)$\Rightarrow$(ii) in Theorem~\ref{tmain} from the Introduction.

\begin{theorem} [{\cite[Theorem~2.2]{AFKKZ-1}}] \label{ccaa}
Let $X$ be an affine variety of dimension $\ge 2$ and ${G\subseteq\SAut(X)}$ be a subgroup generated by a saturated set of locally nilpotent vector fields. Suppose that $G$ acts with the open orbit $O\subseteq X$. Then $G$ acts on $O$ infinitely transitively.
\end{theorem}

For a subset $Z\subseteq X$, we denote by $\cN_Z$ the set of vector fields from $\cN$ that vanish on $Z$. Let
$$
G_{\cN,Z}=\langle H=\exp(\KK\partial): \partial\in \cN_Z \rangle
$$
be the subgroup in $G$ generated by exponents of all elements in $\cN_Z$. It is clear that automorphisms from $G_{\cN,Z}$ preserve the subset $Z$ pointwise. It is easy to see that
$\cN_Z$ is a saturated generating set of the group ~$G_{\cN,Z}$.

One of the main technical results of the work~\cite{AFKKZ-1} can be formulated as follows.

\begin{theorem} [{\cite[Theorem~2.5]{AFKKZ-1}}] \label{ccbb}
Let $X$ be an affine variety of dimension $\ge 2$ and ${G\subseteq\SAut(X)}$ be a subgroup generated by a saturated set $N$ of locally nilpotent vector fields. Suppose that $G$ acts with the open orbit $O\subseteq X$. Then for each finite subset $Z\subseteq O$ the group $G_{N,Z}$ acts transitively on $O\setminus Z$.
\end{theorem}

Theorem~\ref{ccbb} implies Theorem~\ref{ccaa}. In turn, the proof of Theorem~\ref{ccbb} is based on the orbit separation property~\cite[Definition~2.6]{AFKKZ-1}. This proof requires considerable technical efforts.

\smallskip

Finally, let us mention one more result on this subject, which is of independent interest.

\begin{theorem} [{\cite[Theorem~2.15]{AFKKZ-1}}]
Let $X$ be an affine variety of dimension $\ge 2$. Suppose that a group $G\subseteq\SAut(X)$ generated by a saturated set $N$ of locally nilpotent vector fields acts on $X$ with an open orbit $O$ and $Y=X\setminus O$. Then the group $G_{N,Y}$ acts on $O$ infinitely transitively.
\end{theorem}

%%%%%%%%%%%%%%%%%%%%%%%%%
\section{Properties of flexible varieties}
\label{sec-2}

In this section we consider several algebraic and geometric results related to flexibility of an algebraic variety.

%%%%%%%%%%%%%%%%%%
\subsection{The Makar-Limanov invariant and rationality properties} Recall~\cite{Fr} that the \emph{Makar-Limanov invariant} $\text{ML}(X)$ of an affine algebraic variety $X$ is the intersection of kernels of locally nilpotent derivations on $\KK[X]$. In other words, $\text{ML}(X)$ is the subalgebra of $\SAut(X)$-invariant regular functions on $X$.

In~\cite{Li-2}, a close concept is defined: the \emph{Makar-Limanov field invariant} $\text{FML}(X)$ is the intersection of kernels of derivations obtained by extension of locally nilpotent derivations of the algebra $\KK[X]$ to the field of rational functions $\KK(X)$. This is a subfield of the field $\KK(X)$ consisting of rational functions invariant with respect to the group $\SAut(X)$. If this invariant is trivial, i.e. $\text{FML}(X)=\KK$, then the invariant $\text{ML}(X)$ is also trivial, but the converse statement is generally not true. The triviality of $\text{FML}(X)$ is equivalent to the existence of an open $\SAut(X)$-orbit in $X$ (Corollary~\ref{ccdd}) or, equivalently, the existence of a flexible point on~$X$.

It is natural to assume that the triviality of such invariants should be connected with one or another version of rationality properties of the variety $X$. Recall that an irreducible algebraic variety $X$ is \emph{rational} if the field of rational functions $\KK(X)$ is generated over the field $\KK$ by a finite set of algebraically independent elements. Equivalently, there is an open subset in $X$ isomorphic to an open subset of an affine space. An irreducible variety $X$ is \emph{stably rational} if the variety $X\times\AA^k$ is rational for some positive integer $k$. Finally, an irreducible variety $X$ is \emph{unirational} if the field of rational functions $\KK(X)$ can be embedded as a subfield in the field of rational fractions $\KK(x_1,\ldots,x_m)$. Geometrically, this means that there exists a rational dominant morphism to $X$ from an affine space. It is known that if $X$ is a curve or a surface then the rationality of $X$ is equivalent to unirationality. At the same time, there are unirational but not rational varieties of dimension $3$.

Examples of non-unirational three-dimensional affine varieties $X$ with $\text{ML}(X)=\KK$ are known; in these examples, the varieties are birationally equivalent to the direct product $C\times\AA^2$, where $C$ is a smooth curve of genus $g\ge 1$; see~\cite[Example~4.2]{Li-1}. For such varieties, typical $\SAut(X)$-orbits have dimension two, that is, the field invariant 
$\text{FML}(X)$ is nontrivial.

The following proposition confirms, in particular, Conjecture~5.3 from~\cite{Li-2}.

\begin{proposition} [{\cite[Proposition~5.1]{AFKKZ-1}}]
Let $X$ be an irreducible affine variety. If ${\text{FML}(X)=\KK,}$ then the variety $X$ is unirational.
\end{proposition}

The proof follows from the observation that the condition $\text{FML}(X)=\KK$ means that there is an open $\SAut(X)$-orbit $O$ in $X$. Proposition~\ref{ccff} guarantees that there are $\GG_a$-subgroups $H_1,\ldots,H_s$ in $\SAut(X)$ and a point $x\in X$ such that the image of the map
$$
H_1\times\ldots\times H_s \to X, \quad (h_1,\ldots,h_s) \mapsto (h_1\ldots h_s). x
$$
equals $O$. Since the product $H_1\times\ldots\times H_s$ is isomorphic (as a variety) to the affine space~$\AA^s$, this implies unirationality of $X$. Moreover, every two points on $O$ 
are connected by the image of a morphism $\AA^1\to O$. In other words, the orbit $O$ is \emph{$\AA^1$-connected}.

In general, flexibility entails neither rationality nor stable rationality. For example, Popov~\cite[Example~1.22]{Po2} pointed out examples of finite subgroups $F\subseteq\SL_n$ for some $n$ such that the smooth flexible unirational affine homogeneous space $X=\SL_n/F$ is not stably rational.

\smallskip

The paper~\cite{BKK} is devoted to a conjecture that characterizes unirationality in terms of infinitely transitive actions. To formulate the conjecture, we need several definitions.

We say that an irreducible variety $X$ is \emph{stably birationally infinitely transitive} if for some positive integer $m$ the purely transcendental extension $\KK(X)(y_1,\ldots,y_m)$ of the field $\KK(X)$ is the field of rational functions on some flexible affine variety of dimension $\ge 2$. If this condition holds for $m=0$, the variety $X$ is called \emph{birationally infinitely transitive}.

It is clear that the affine space $\AA^n$ is stably birationally infinitely transitive, and for $n\ge 2$ it is birationally infinitely transitive. Thus, the same properties hold for any rational algebraic variety.

\begin{conjecture} [{\cite[Conjecture~1.4]{BKK}}] \label{cbkk}
Any unirational variety is stably birationally infinitely transitive. 
\end{conjecture}

In~\cite[Example~1.6]{BKK} it is shown that the stability requirement is essential for the conjecture. In fact, the authors show that any unirational three-dimensional variety admitting an infinitely transitive birational model is rational. Indeed, for a typical $\GG_a$-orbit $C$ on a given model, this model is birationally isomorphic to $C\times Y$ for some surface $Y$. By assumptions, the surface $Y$ is unirational, and therefore rational. At the same time, there are examples of unirational but not rational three-dimensional varieties.

\smallskip

In ~\cite[Theorems~2.1--2.2]{BKK}, sufficient conditions for stable birational infinite transitivity of a variety are given. These conditions allow to prove Conjecture~\ref{cbkk} for wide classes of varieties, see~\cite[Section~3]{BKK}.

%%%%%%%%%%%%%%%%%%%%%

\subsection{Collective transitivity} By \emph{\it collective infinite transitivity} we mean the possibility to translate simultaneously, that is, by the same automorphism, arbitrary finite sets of points to arbitrary finite sets of points of the same cardinality in the same orbits. The general results on collective infinite transitivity obtained in~\cite[Section~3]{AFKKZ-1} are illustrated here by examples from matrix algebra, see~\cite{Re-1,Re-2}.

Let $X=\Mat(n,m)$ be the space of matrices of size $n\times m$ over the ground field ~$\KK$. It is well known that a locally closed subset $X_r\subseteq X$ of matrices of rank $r$ has dimension $mn-(m-r)(n-r)$. Further we assume that this dimension is at least~$2$. The group $\SL_n\times\SL_m$ acts on $X$ by means of left/right multiplications and preserves the subsets~$X_r$. For any $k\ne l$, let $E_{kl}$ and $E^{kl}$ be matrix units, that is, $x_{kl}=1$ and all other matrix elements are zero. Let $H_{kl}=I_n+\KK E_{kl}\subseteq\SL_n$ and $H^{kl}=I_m+\KK E^{kl}\subseteq\SL_m$ denote the corresponding $\GG_a$-subgroups. They preserve the stratification $X=\bigcup_r X_r$. Let $\delta_{kl}$ and $\delta^{kl}$ be the corresponding locally nilpotent vector fields on $X$. These fields are tangent to the elements of the stratification.

Let us say that the $\GG_a$-subgroups $H_{kl}$, $H^{kl}$ and their replicas are \emph{elementary} $\GG_a$-subgroups. Our goal is to establish collective infinite transitivity on strata for the subgroup $G$ of the group $\SAut(X)$ generated by such two-sided elementary subgroups.

We know from linear algebra that the subgroup $\SL_n\times\SL_m\subseteq G$ acts transitively on every subvariety $X_r$ with the exception of the open subvariety $X_n$ in the case $m=n$. In the latter case, $G$-orbits lying in $X_n$ are the level subvarieties of the determinant.

\begin{theorem} [{\cite[Theorem~3.3]{AFKKZ-1}}]
Let us consider two finite ordered sets $S_1$ and $S_2$ of the same cardinality, consisting of such matrices from $\Mat(n,m)$ that the sequences of ranks of these matrices coincide. In the case $m=n$, we additionally require a matching of the sequences of determinants. Then there is an element of the group $G$ whose diagonal action translates $S_1$ to $S_2$.
\end{theorem}

Similar results in the case of symmetric and skew-symmetric matrices are obtained in ~\cite[Section~3.3]{AFKKZ-1}.

%%%%%%%%%%%%%%%%%%%%%%%%%%%%
\subsection{$\AA^1$-richness and the Gromov-Winkelmann Theorem} Let $X$ be a flexible affine variety of dimension $\ge 2$. Consider a set of pairwise distinct points $p_1,\ldots,p_k\in X^\reg$. Let us fix a $\GG_a$-orbit $C$ on $X^\reg$ and a set of $k$ pairwise distinct points ${q_1,\ldots,q_k\in C}$. It follows from infinite transitivity that there is such an element
$g\in\SAut(X)$ that $g\cdot q_1=p_1,\ldots,g\cdot q_k=p_k$. Thus, the shift $g\cdot C$ of the curve $C$ is a $\GG_a$-orbit on $X$ with respect to the conjugated action and it passes through the points $p_1,\ldots,p_k$. This elementary observation can be strengthened as follows.

An algebraic variety $X$ is called \emph{$\AA^1$-rich} if for any finite subset $Z$ and any closed subvariety $Y$ of codimension $\ge 2$ that does not intersect $Z$, there is a curve on $X$ that is isomorphic to the line $\AA^1$, does not intersect $Y$, and passes through all points of $Z$. The next result follows from the Transversality Theorem (Theorem~\ref{ccgg}).

\begin{proposition}[{\cite[Corollary~4.18]{AFKKZ-1}}]
Let us consider an affine variety $X$ and assume that the group $\SAut(X)$ acts with an open orbit~${O\subseteq X}$. Then for any finite subset $Z\subseteq O$ and any closed subvariety $Y\subseteq X$ of codimension $\ge 2$ with the condition $Z\cap Y=\varnothing$ there is an orbit $C\cong\AA^1$ of some $\GG_a$-action on $X$ which does not intersect $Y$ and passes through all points of the subset $Z$
\end{proposition}

In the case $X=\AA^n$, this result follows from the Gromov-Winkelmann Theorem~\cite{Wi}. This theorem states that the group $\SAut (\AA^n\setminus Y)$ acts on
$\AA^n\setminus Y$ transitively and, as we know, this entails infinite transitivity of the action.

\smallskip

A generalization of the Gromov-Winkelmann Theorem is obtained in ~\cite{FKZ}. Firstly, the authors show that the flexibility of an affine variety $X$ is equivalent to the flexibility of the smooth quasi-affine variety $X^\reg$.

\begin{theorem}[{\cite[Theorem~0.1]{FKZ}}]
Let $X$ be a smooth flexible quasi-affine variety of dimension $\ge 2$ and $Y\subseteq X$ be a closed subscheme of codimension $\ge 2$. Then the variety $X\setminus Y$ is also flexible.
\end{theorem} 

Finally, it is shown in ~\cite[Proposition~1.8]{FKZ} that for a normal quasi-affine variety $X$ and a closed subset $Y\subseteq X$ of codimension $\ge 2$, each $\GG_a$-action on $X\setminus Y$ can be extended to a $\GG_a$-action on $X$, which stabilizes~$Y$.

%%%%%%%%%%%%%%%%%%%%%%
\subsection{Transitivity on jets} We limit ourselves in this subsection to the formulation of only one result. By the volume form on a variety~$X$ of dimension $n$, we mean a differential $n$-form $\omega$ defined on $X^\reg$ that does not vanish anywhere.

\begin{theorem} [{\cite[Theorem~4.14 and Remark~4.16]{AFKKZ-1}}]
Let $X$ be a flexible affine variety of dimension $n\ge 2$ equipped with a volume form~$\omega$. Then for any $m\ge 0$ and any finite subset
$Z\subseteq X^\reg$ there is an automorphism $g\in\SAut(X)$ with prescribed $m$-jets at points $p\in Z$ provided that these jets preserve the form $\omega$ and map $Z$ to $X^{\reg}$  injectively. The same holds without assuming that there is a global volume form on $X^{\reg}$ provided that each point $p\in Z$ corresponds to a jet that fixes the point $p$ and the linear part of the jet belongs to $\SL(T_pX)$.
\end{theorem}

Other results on transitivity on jets can be found in~\cite[Section~4.2]{AFKKZ-1}. 

%%%%%%%%%%%%%%
\subsection{Multiple transitivity of the automorphism group} It is natural to ask whether the flexibility is the only reason that the automorphism group $\Aut(X)$ of an affine variety $X$ acts on the smooth locus of $X$ multiply transitively. In~\cite{Ar} a positive answer to this question is obtained under some additional restrictions.

We say that a variety $X$ satisfies \emph{condition~$(*)$} if $X$ admits a non-identical action of a connected affine algebraic group of positive dimension. Condition~$(*)$ is equivalent to the fact that there is a non-identical action of either the group $\GG_a$ or the group $\GG_m$ on~$X$.

\begin{theorem}[{\cite[Theorem~11]{Ar}}] 
Let $X$ be an irreducible quasi-affine variety of dimension $\ge 2$ satisfying condition~$(*)$. Suppose that the group $\Aut(X)$ acts on $X$ with an open orbit~$O$. Then the following conditions are equivalent.
{\allowdisplaybreaks
\begin{enumerate}
\item[(1)]
The group $\Aut(X)$ acts on $O$ $2$-transitively. 
\item[(2)]
The group $\Aut(X)$ acts on $O$ infinitely transitively. 
\item[(3)]
The group $\SAut(X)$ acts on $O$ transitively. 
\item[(4)]  
The group $\SAut(X)$ acts on $O$ infinitely transitively. 
\end{enumerate}
}
\hspace{-2mm}
In particular, the group $\Aut(X)$ acts on the smooth locus $2$-transitively if and only if the variety $X$ is flexible.
\end{theorem}

Implications $(4)\Rightarrow (2)$ and $(2)\Rightarrow (1)$ are obvious. Implication $(3)\Rightarrow (4)$ follows from Theorem~\ref{ccaa}. 

It remains to prove implication $(1)\Rightarrow (3)$. Here we give arguments that are somewhat different from the arguments in~\cite{Ar} and are based on the group-theoretic facts set out in the second appendix to this paper. Let us first assume that there is a non-identical action of the group $\GG_a$ on $X$. This means that the group $\SAut(X)$ is nontrivial. Since $\SAut(X)$ is a normal subgroup of the group $\Aut(X)$, condition~$(1)$ and Lemma~\ref{prtr} imply condition~$(3)$.

Let us recall that an algebraic variety $X$ is called \emph{rigid} if there is no non-identical $\GG_a$-action on $X$. For an affine $X$, this condition is equivalent to the fact that the algebra $\KK[X]$ does not admit nonzero locally nilpotent derivation. In this case, there is a unique maximal torus $T$ in the group $\Aut(X)$. This result is proved in~\cite[Theorem~1]{AG} for an affine variety $X$ and is generalized to the quasi-affine case in~\cite[Theorem~5]{Ar}. In particular, the maximal torus $T$ is a normal subgroup in $\Aut(X)$. According to condition~$(*)$, the torus $T$ is nontrivial. From Lemma~\ref{prtr} we conclude again that $T$ acts transitively on $O$. Thus, the orbit $O$ is isomorphic to the torus~$T$. However, the automorphism group 
of the variety $T$ is a semidirect product of $T$ and the discrete group $\GL_n(\ZZ)$. Such a group acts transitively but not $2$-transitively on $T$. Thus, for a rigid variety $X$ with condition~$(*)$ condition~$(1)$ does not hold, and implication $(1)\Rightarrow(3)$ is proved.

\smallskip

The ideas of this reasoning can be used in the following situation.

\begin{definition} 
We call a set of automorphisms of an algebraic variety $X$ \emph{geometric} if this set is invariant under conjugation by any automorphism of the variety~$X$ and contains at least one non-identical automorphism.
\end{definition} 

Informally speaking, geometric sets of automorphisms can be understood as sets that are defined by certain invariant geometric properties. For example, such sets form all automorphisms of order two that have a single fixed point, or all automorphisms that preserve the set of curves that do not intersect, cover the smooth locus of the variety $X$, and each of these curves is isomorphic to the affine line.

\begin{proposition} \label{relnew}
Let $X$ be a flexible quasi-affine variety of dimension $\ge 2$. Then for any geometric set of automorphisms of the variety $X$, the subgroup of the automorphism group generated by this set  acts on the smooth locus $X^{\reg}$ infinitely transitively.
\end{proposition} 

The proof follows immediately from the fact that a geometric set of automorphisms generates a nontrivial normal subgroup in $\Aut(X)$, and Theorem~\ref{norgr} shows that this group is infinitely transitive on $X^{\reg}$.

\smallskip

In~\cite{Po} the author studies relations between multiple transitivity of the action of a group of automorphisms and geometric properties of a variety, such as rationality or unirationality. An action of a group $G$ on a variety $X$ is said to be \emph{generically $m$-transitive} if the restriction of this action to some open invariant subset of the variety $X$ is $m$-transitive. In~\cite[Theorem~5]{Po} it is proved over an uncountable algebraically closed field that for any irreducible variety $X$, the condition of generic {$2$-transitivity} of the action of the group $\Aut(X)$ implies that either $X$ is unirational, or the group $\Aut(X)$ does not contain nontrivial connected algebraic subgroups. In particular, if the variety $X$ is complete, the generic $2$-transitivity of the action of $\Aut(X)$ on $X$ implies unirationality.

\smallskip

We also note that flexible varieties form a natural class of varieties where it makes sense to investigate the problem of continuation of isomorphisms between subvarieties. Namely, let $X$ be a smooth flexible quasi-affine variety and $Y_1,Y_2$ be two closed subvarieties in $X$. The extension problem consists in finding conditions on the variety $X$, the subvarieties $Y_1,Y_2$, and possibly tangent subbundles $TY_1,TY_2$, under which any isomorphism $Y_1\to Y_2$ can be extended to an automorphism of the variety $X$. A classical result of this type is the Abyankar-Moh-Suzuki Theorem, which states that the images of any two isomorphic embeddings of the affine line $\AA^1$ into the affine plane $\AA^2$ can be sent to each other by an automorphism of the plane. A number of sufficient conditions for extension of an isomorphism between subvarieties in a flexible variety can be found in~\cite{Ka,KU} and in the works listed there in the references.

%%%%%%%%%%%%%%%%%%%%%%%%%%%%%%%%%%%
\section{Examples of flexible varieties}
\label{sec-3}

Considerable interest in flexible varieties is caused not only because such varieties have a number of remarkable geometric properties, but also by the fact that the class of flexible varieties is surprisingly wide. In this section we give various examples of flexible varieties. We largely follow the work plan from~\cite{AFKKZ-2}, supplementing the presentation with the results of recent years.

%%%%%%%%%%%%%%%%%
\subsection{Suspensions} One of the first papers, where the property of infinite transitivity of the action of a group of automorphisms of an algebraic variety was studied, is the work~\cite{KZ}.

\begin{theorem} [{\cite[Theorem~5.1]{KZ}}] \label{suspkz}
Consider a hyperplane $X$ in the affine space~$\AA^n,$ $n\ge 4,$ given by an equation of the form 
$$
x_1x_2=f(x_3,\ldots,x_n),
$$
where $f$ is an arbitrary non-constant polynomial in $n-2$ variables. Then the group $\SAut(X)$ acts on the smooth locus of the variety $X$ infinitely transitively. 
\end{theorem}

This result motivates the following definition. 

\begin{definition}
Let $X$ be an irreducible affine variety and $f\in\KK[X]$ be a non-constant function. The \emph{suspension} over the variety $X$ is the affine variety
$$
\Susp(X,f)=\{uv-f(x)=0\}\subseteq\AA^2 \times X,
$$
where $u$ and $v$ are coordinates on the plane $\AA^2$. 
\end{definition}

In these terms, the varieties from Theorem~\ref{suspkz} are exactly suspensions over the affine space~$\AA^{n-2}$. The following general fact holds.

\begin{theorem}[{\cite[Theorem~3.2]{AKZ-1}}] \label{suspakz}
Let $X$ be a flexible affine variety and $f\in\KK[X]$ be a non-constant function. Then the suspension $\Susp(X,f)$ is also a flexible variety.
\end{theorem}

Thus, iterating the suspension construction, we obtain a wide class of flexible varieties.

Note that in Theorem~\ref{suspakz}, the case of a one-dimensional variety $X$ is also allowed. In this case, we have $X=\AA^1$ automatically. In fact, a similar result for the affine line can be obtained over any field of characteristic zero.

\begin{theorem}[{\cite[Theorem~3.1]{AKZ-1}}] \label{suspa1}
Let $\KK$ be a field of characteristic zero and $f\in\KK[x]$ be a polynomial such that $f(\KK)=\KK$. Consider a hypersurface $X\subseteq\AA^3$ given by the equation $x_1x_2=f(x_3)$.
Then the group $\SAut(X)$ acts on the smooth locus of the variety $X$ infinitely transitively. 
\end{theorem}

The case of the field of real numbers $\KK=\RR$ was considered separately. The initial result~\cite[Theorem~3.3]{AKZ-1} has been strengthened to the following general fact.

\begin{theorem}[{\cite[Theorem~1]{KM}}] 
Let $X$ be an affine algebraic variety over the field of real numbers and $f\in\RR[X]$. Suppose that each connected component $X^i$ of the smooth locus of the variety $X$ has dimension $\ge 2$ and $f$ is non-constant on~$X^i$. If the variety $X$ is flexible and the action of the group $\SAut(X)$ is infinitely transitive on each component~$X^i$, then the same properties hold for the suspension $\Susp(X,f)$.
\end{theorem}

Let us return to the case of an algebraically closed field of characteristic zero and describe several more classes of flexible affine hypersurfaces. We fix positive integers $n_0, n_1, n_2$ and put ${n\!=\!n_0\!+\!n_1\!+\!n_2}$. For each $i=0,1,2$, we also fix a set of positive integers $l_i\in\ZZ^{n_i}_{>0}$ and define monomials
$$
T_i^{l_i}:= T_{i1}^{l_{i1}}\ldots T_{in_i}^{l_{in_i}} \in \KK[T_{ij} \, ; \, i=0,1,2, \ j=1,\ldots,n_i]. 
$$
A \emph{trinomial} is a polynomial of the form $f=T_0^{l_0}+T_1^{l_1}+T_2^{l_2}$, and a \emph{trinomial hypersurface} $X$ is the set of zeroes of a trinomial $f$
in the affine space $\AA^n$. It is not difficult to verify that $X$ is an irreducible normal affine variety of dimension $n-1$. Among such hypersurfaces there are examples of both rigid and flexible varieties. For example, if for some $i$ we have $n_i=l_{i1}=1$, then the trinomial hypersurface is isomorphic to the affine space $\AA^{n-1}$, and if $n_i=2$ and $l_{i1}=l_{i2}=1$, then the trinomial hypersurface is a suspension over the affine space $\AA^{n-2}$.
In addition to these examples, five sufficiently broad classes of flexible trinomial hypersurfaces are described in~\cite[Theorem~4]{Ga}. For instance, hypersurfaces corresponding to trinomials of the form $T_0^{l_0}+T_1^{l_1}+T_{21}T_{22}\ldots T_{2n_2}$ are flexible. At the same time, in~\cite{Ga} flexibility and rigidity are investigated both for trinomial hypersurfaces and for a wider class of affine trinomial varieties.

%%%%%%%%%%%%%%%%%%%%%
\subsection{Affine toric varieties} \label{3.2} 
We recall that a normal algebraic variety $X$ is \emph{toric} if there exists a regular effective action of an algebraic torus $T$ on $X$ with an open orbit. Some affine toric varieties are not flexible. For example, if we put $X=T$, then the algebra $\KK[X]$ is generated by invertible functions, and the group $\SAut(X)$ is trivial.

Let us call a toric variety \emph{non-degenerate} if every invertible regular function on $X$ is constant. This condition is equivalent to the fact that $X$ cannot be represented as a direct product $X'\times\GG_m$ for some toric variety $X'$. 

\begin{theorem}[{\cite[Theorem~2.1]{AKZ-1}}] \label{tna}
Evere non-degenerate affine toric variety is flexible.
\end{theorem}

In order to explain the idea of the proof of this result, we consider $T$-normalized $\GG_a$-subgroups and describe such subgroups in terms of so-called Demazure roots.

Let $N$ be the lattice of one-parameter subgroups of the torus $T$ and $M$ be the dual lattice of characters. Let $\langle\,\cdot\,,\,\cdot\,\rangle\colon N\times M\to \ZZ$ be the canonical pairing. We let $\chi^m$ denote the character of the torus $T$ corresponding to a point $m\in M$. In this case we have $\chi^m\chi^{m'}\, =\, \chi^{m+m'}$, so the group algebra
$$
\KK[M] : =  \bigoplus_{m \in M} \KK\chi^m
$$
can be identified with the algebra $\KK[T]$ of regular functions on the torus $T$. Denote by $T . x_0$ the open $T$-orbit on $X$. Since the orbit map $T\to X$,
$t\longmapsto t . x_0$ is dominant, we can identify $\KK[X]$ with a subalgebra in $\KK[M]$. More precisely, there is such a convex polyhedral cone $\sigma^\vee\subseteq M_{\QQ} := M\otimes_{\ZZ}\QQ$ that $\KK[X]$ coincides with the semigroup algebra of the semigroup $\sigma^\vee\cap M$, that is
$$
\KK[X]  =  \bigoplus_{m \in \sigma^\vee \cap M} \KK\chi^m.
$$
The dual cone $\sigma$ to the cone~$\sigma^\vee$ is a strictly convex cone in $N_{\QQ}$. We assume that the variety $X$ is non-degenerate. This means that the cone $\sigma$ does not lie in a proper subspace of the space $N_{\QQ}$. Let $\Xi=\{p_1,\ldots,p_r\}$ be the set of generators on rays, that is, the set of primitive vectors on one-dimensional faces of the cone~$\sigma$. Each vector $p\in N$ is associated with a one-parameter subgroup $R_p$ of the torus~$T$.

\smallskip

There are natural one-to-one correspondences between faces $\delta$ of the cone $\sigma$, the dual faces $\tau=\delta^\bot$ of the cone $\sigma^\vee$, and $T$-orbits $O_\tau$ on $X$, where $\dim O_\tau=\dim\tau=\dim\sigma-\dim\delta$. In particular, the only $T$-fixed point on $X$ corresponds to the vertex of the cone~$\sigma^\vee$, and the open $T$-orbit corresponds to the cone $\sigma^\vee$ itself. These correspondences preserve inclusions: a $T$-orbit $O_\mu$ is contained in the closure $\overline{O_\tau}$ if and only if $\mu\subseteq\tau$ or, equivalently, $\mu^\bot\supseteq\tau^\bot$.

Every face $\tau\subseteq\sigma^\vee$ defines a decomposition
$$
\KK[X] \ =\KK[\overline{O_\tau}]\oplus I(\overline{O_\tau}),
$$
where
$$
\KK[\overline{O_\tau}]=\bigoplus_{m \in\tau\cap M} \KK\chi^m,\qquad\mbox{and}\qquad I(\overline{O_\tau})=\bigoplus_{m \in (\sigma^\vee\setminus \tau) \cap M} \KK\chi^m
$$
is the ideal of the subvariety $\overline{O_\tau}$ in $\KK[X]$.

The stabilizer $T_{p}$ of a point $p\in X$ is connected, that is, $T_{p}\subseteq T$ is a subtorus. Further, we have an inclusion $T_p\subseteq T_q$ exactly when $T.q\subseteq\overline{T.p}$, and ${\overline{T.p}=X^{T_p}}$, where $X^{G}$ denotes the set of fixed points of an action of a group $G$ on a set~$X$.

\begin{definition}[{\cite{De}}]
A \emph{Demazure root} of a cone $\sigma$ is a vector ${e\in M}$ such that for some $i$ with $1\le i\le r$, where $r$ is the number of elements in $\Xi$, we have
$$
\langle p_i, e\rangle \, = \, -1 \quad \text{and} \quad \langle p_j, e\rangle\ge 0 \quad \text{for all} \quad j\ne i.
$$
\end{definition}

Denote by ${\mathcal R}(\sigma)$ the set of all Demazure roots of a cone $\sigma$. There is a one-to-one correspondence $e\longleftrightarrow H_e$ between roots of the cone $\sigma$ and $\GG_a$-subgroups in $\Aut(X)$ normalized by the acting torus, see ~\cite{De,Li-1}. Such subgroups are called \emph{root subgroups}.

We put $p_e := p_i$. A root $e\in {\mathcal R}(\sigma)$ defines a locally nilpotent derivation $\partial_e$ on the $M$-graded algebra $\KK[X]$ given as
$$
\partial_e(\chi^m) \ = \ \langle p_e,m \rangle \chi^{m+e}.
$$
Its kernel is a finitely generated homogeneous subalgebra in $\KK[X]$:
$$
\ker\partial_e=\bigoplus_{m\in p_e^\bot\cap M}\KK\chi^m,
$$
where $p_e^\bot = \{m\in\sigma^\vee\cap M, \, \langle p_e,m\rangle=0\}$ is a facet, i.e. a face of codimension one, of the cone $\sigma^\vee$ orthogonal to~$p_e$.

Two roots $e$ and $e'$ such that $p_e=p_{e'}$ are called {\em equivalent}; we write $e\sim e'$. The roots $e$ and $e'$ are equivalent if and only if
$\ker\partial_e=\ker\partial_{e'}$.

A numbering of the generators on the rays $\Xi=\{p_1,\ldots,p_r\}$ determines the partition
$$
{\mathcal R}(\sigma)=\bigcup_{i=1}^r {\mathcal R}_i,
\quad\mbox{where}\quad {\mathcal R}_i =\{e\in {\mathcal
R}(\sigma)\,\vert\,p_e=p_i\}\,
$$
are nonempty. Moreover, these subsets are infinite if the dimension of the cone is~$\ge 2$. 

As an example, we consider the affine plane $X=\AA^2$ with the standard action of a two-dimensional torus. Here the cones $\sigma$ and $\sigma^\vee$ coincide with the positive quadrants. The set ${\mathcal R}(\sigma)$ consists of two equivalence classes
$$
{\mathcal R}_1=\{(-1,y)\,\vert\, y\in\ZZ_{\ge 0}\}\quad\mbox{and}\quad
 {\mathcal R}_2=\{(x,-1)\,\vert\, x\in\ZZ_{\ge 0}\}.
$$

\smallskip

The derivation $\partial_e$ generates a $\GG_a$-subgroup
$$
H_e=\lambda_e(\KK) \subseteq\Aut(X),
$$
where $\lambda_e\colon t\longmapsto \exp(t\partial_e)$. The algebra of invariants $\KK[X]^{H_e}$ coincides with $\ker\partial_e$. The embedding $\KK[X]^{H_e}\subseteq\KK[X]$ induces a morphism $\pi\colon X\to Z=\Spec\KK[X]^{H_e}$ whose typical fibers are one-dimensional $H_e$-orbits isomorphic to the line~$\AA^1$. The torus $T$ normalizes the subgroup $H_e$. In particular, $T$ leaves the set of fixed points $X^{H_e}$ invariant.

Let $R_e=R_{p_e}\subseteq T$ be the one-dimensional subtorus corresponding to the vector ${p_e\in N}$. The action of $R_e$ on the graded algebra $\KK[X]$ with a suitable parametrization $p_e:\GG_{\text{m}}(\KK)\ni t\longmapsto p_e(t)\in R_e$ is given by 
$$
t.\chi^m \ = \ t^{\langle p_e,m \rangle} \chi^{m},\quad t\in \GG_{\text{m}}(\KK).
$$

In particular, we have $\KK[X]^{R_e}=\KK[X]^{H_e}$. Therefore, the morphism ${\pi:X\to Z}$ coincides with the quotient morphism $X\to X/\!/R_e$, and typical $H_e$-orbits coincide with the closures of typical $R_e$-orbits. This statement holds for any one-dimensional $H_e$-orbit.

We have a decomposition
$$
\KK[X]=\KK[X]^{R_e} \ \oplus \ \bigoplus_{m \in \sigma^\vee\cap M \setminus p_e^\bot} \KK \chi^m =\KK[X]^{R_e}\oplus I(D_e),
$$
where $D_e:=X^{R_e}\cong Z$. The divisor $D_e$ coincides with the set of limit points of the action of $R_e$ on $X$. Thus, each one-dimensional $R_e$-orbit has a limit point on $D_e$.

\begin{proposition}[{\cite[Proposition~2.1]{AKZ-1}}]
\label{pr1}
Let $e\in {\mathcal R}(\sigma)$ and $H_e\subseteq\SAut (X)$ be the associated $\GG_a$-subgroup. Then the following statements hold. 
\begin{enumerate}\item[(a)]
For any point $x\in X \setminus X^{H_e}$ the orbit $H_e. x$ intersects precisely two $T$-orbits $\Of_1$ and $\Of_2$ on $X,$ and $\dim \Of_1 = \dim \Of_2 +1$.
\item[(b)] 
The intersection $\Of_2\cap H_e. x$ consists of a unique point, while
$$
\Of_1\cap H_e. x = R_e. y,\quad  y \in \Of_1\cap H_e. x.
$$\end{enumerate}
\end{proposition}

We say that a pair of $T$-orbits $(\Of_1, \Of_2)$ on $X$ is {\it$H_e$-connected} if $H_e . x\subseteq \Of_1\cup\Of_2$ for some $x\in X\setminus X^{H_e}$. Proposition~\ref{pr1} shows that for such a pair (up to permutation) we have $\Of_2\subseteq\overline{\Of_1}$ and $\dim\Of_1=\dim\Of_2+1$. Here we choose the point~$x$ on the orbit ~$\Of_1$. Since the torus normalizes the subgroup $H_e$, any point on $\Of_1$ can serve as the point $x$ from the definition of $H_e$-connectedness.

Proposition~\ref{pr1} implies the following criterion of $H_e$-connectedness.

\begin{lemma}[{\cite[Lemma~2.2]{AKZ-1}}] \label{l1}
Let $(\Of_1,\Of_2)$ be a pair of $T$-orbits on $X$ with ${\Of_2\subseteq\overline{\Of_1},}$ where $\Of_i=\Of_{\sigma_i^\bot}$ for a face $\sigma_i$ of the cone $\sigma,$
$i=1,2$. For a given root $e\in {\mathcal R}(\sigma)$ a pair $(\Of_1,\Of_2)$ is $H_e$-connected if and only if we have $e\vert_{\sigma_2}\!\le\! 0$ and $\sigma_1$ is a facet of the cone $\sigma_2$ given by the equation $\langle v,e\rangle=0$.
\end{lemma}

The following lemma is a key technical result.

\begin{lemma}[{\cite[Lemma~2.3]{AKZ-1}}] \label{l2} 
For any point $x\in X^{\reg} \setminus \Of_{\sigma^\vee}$ there is a root $e\in {\mathcal R}(\sigma)$ such that
$$
\dim T.y \ > \ \dim T. x
$$
for a generic point $y\in H_e.x$. In particular, the pair $(T.y,T. x)$ is $H_e$-connected.
\end{lemma}

From this lemma it is not difficult to deduce the transitivity of the action of the group $\SAut(X)$ on the smooth locus of a non-degenerate affine toric variety $X$, which leads to the proof of Theorem~\ref{tna}.

\begin{remark}
Already in the case of affine toric surfaces, we can give examples of flexible varieties $X$ such that the smooth locus $X^{\reg}$ is not a homogeneous space for any algebraic group, see~\cite[Example~2.2]{AKZ-1}.
\end{remark}

To study actions of $\GG_a$-subgroups, the following generalization of the notion of a root subgroup on a toric variety is very useful.

\begin{definition}
Let $X$ be an algebraic variety with an action of an algebraic torus~$T$. A $\GG_a$-subgroup in $\Aut(X)$ is called a \emph{root subgroup} if it is normalized by the torus~$T$. In this situation, the character $e$ of the torus $T$ by which $T$ acts by conjugation on the $\GG_a$-subgroup is called the \emph{root} of the corresponding subgroup:
$$
ts(a)t^{-1}=s(\chi^e(t)a), \quad t\in T, \quad a\in\KK.
$$
\end{definition} 

Many research projects are devoted to the study of root subgroups of automorphism groups or, equivalently, homogeneous locally nilpotent derivations of graded algebras. This topic deserves a special survey, and we will not consider it in detail here.

\smallskip

Finally, we note that in the recent paper~\cite{BG}, flexibility and rigidity criteria for non-normal affine toric varieties are obtained.

%%%%%%%%%%%%%%
\subsection{Homogeneous and almost homogeneous varieties} Following~\cite{Po2}, we consider a class of connected linear algebraic groups that are generated by their $\GG_a$-subgroups. It is not difficult to see that this class coincides with the class of connected linear algebraic groups that do not admit nontrivial characters. In turn, these are exactly the groups that are semidirect products of a connected semisimple group and a unipotent radical. If such a group $G$ acts on a variety $X$, then the image of $G$ in $\Aut(X)$ is contained automatically in the group $\SAut(X)$. In particular, if $G$ acts on the smooth locus of a variety~$X$ transitively, then $X$ is flexible.

As an example of such a situation, we consider a simple finite-dimensional module $V$ of a semisimple group $G$ and let $X$ be equal to the closure of the orbit of a highest weight vector.
It is known from~\cite{VP} that the variety $X$ consists of two orbits, an open orbit and a fixed point coinciding with the origin. If the variety $X$ does not coincide with the space $V$, the origin is a singular point of the variety $X$. This shows that $X$ is flexible. The variety $X$ is a normal affine cone over the generalized flag variety $G/P$, where $P$ is a parabolic subgroup of the group $G$. The flexibility of such cones was first proved in~\cite[Theorem~1.1]{AKZ-1}.

Also, the above proposition can be applied to homogeneous spaces $G/H$ of groups $G$ without nontrivial characters. It is well known that for a semisimple group $G$ such a homogeneous space is affine if and only if the subgroup $H$ is reductive~\cite[Theorem~4.17]{VP-1}. Along the way we obtain many important examples of smooth flexible varieties \cite[Proposition~5.4]{AFKKZ-1}. For instance, such examples are homogeneous spaces of the groups $\SL_n$, $\SO_n$ and $\Sp_{2n}$.

\smallskip

Now we suppose that a connected linear algebraic group $G$ acts on an irreducible variety $X$ with an open orbit. In this situation, the variety $X$ is called \emph{almost homogeneous}. It turns out that under certain conditions it is possible to guarantee the flexibility of such varieties.

%%%%%%%%%%%%%%%%%%
\subsubsection{The smooth case.} Suppose that an almost homogeneous affine variety $X$ is smooth. In the case of a semisimple group $G$, using Luna's Slice Theorem, we show in~\cite[Theorem~5.6]{AFKKZ-1} that $X$ is a homogeneous variety with respect to a larger group which is a semidirect product $G\rightthreetimes V$, where $V$ is a finite-dimensional $G$-module. Since the group $G\rightthreetimes V$ has no characters, the variety~$X$ is flexible.

In the case of an arbitrary reductive group $G$ under the assumption that there is no non-constant invertible regular function on the variety, a similar result is proved in~\cite[Theorem~2]{GS}.

%%%%%%%%%%%%%%%%%%
\subsubsection{$\SL_2$-embeddings.} Suppose that the group $\SL_2$ acts with an open orbit on a normal affine three-dimensional variety~$X$. The classification of such varietiess was obtained by Popov~\cite{Po-5} in the early 70s. In this situation, the stabilizer of a general position for the action of the group $\SL_2$ is finite. If such a stabilizer is not commutative, the variety turns out to be $\SL_2$-homogeneous, and we come to the situation considered above. In the remaining cases, the stabilizer of a general position is cyclic, and two cases are possible here. In the first case, the variety consists of no more than two $\SL_2$-orbits, it is smooth and therefore flexible.
In the second case, the variety $X$ consists of three orbits, one is a fixed point, and it is the only singular point on $X$. In~\cite{BH} it is shown that $X$ is realized as a categorical quotient of the affine hypersurface $x_0^b=x_1 x_4-x_2 x_3$ by the action of a one-dimensional diagonalizable group. This hypersurface is a suspension over $\AA^3$. Using this circumstance, in~\cite[Theorem~5.7]{AFKKZ-1} we construct a $\GG_a$-subgroup in $\Aut(X)$ that connects three-dimensional and two-dimensional $\SL_2$-orbits on $X$, and thereby prove the flexibility of the variety $X$.

As a result, we obtain that any normal affine variety with an action of the group $\SL_2$ with an open orbit is flexible.

%%%%%%%%%%%%%%%%%%%%%
\subsubsection{Horospherical and spherical varieties} We recall that a normal algebraic variety $X$ on which a connected reductive group $G$ acts with an open orbit is \emph{horospherical} if the stabilizer in $G$ of a point in the open orbit contains a maximal unipotent subgroup of~$G$. Affine horospherical varieties were studied in detail in~\cite{VP}.

The flexibility of affine horospherical varieties of a semisimple group $G$ is proved in~\cite{Sh}. The idea of the proof is to construct $\GG_a$-subgroups in $\Aut(X)$ that connect different $G$-orbits on $X$ consisting of smooth points. These subgroups are realized as exponents of locally nilpotent derivations whose kernels, as in the toric case, correspond to the faces of a cone associated with a horospherical variety.

In the case of an arbitrary reductive group $G$, the flexibility of affine horospherical varieties without non-constant invertible regular functions is proved in~\cite[Theorem~3]{GS}. In particular, a new proof of the flexibility of non-degenerate affine toric varieties is obtained here.

\smallskip

A natural generalization of horospherical varieties are spherical varieties. Recall that a normal algebraic variety $X$ with an action of a reductive group $G$ is called \emph{spherical} if the restriction of the action to a Borel subgroup $B$ has an open orbit in $X$. Spherical varieties are considered as the most adequate analogue of toric varieties in the case of actions of noncommutative reductive groups. There is also a combinatorial, albeit much more complex description of spherical varieties in terms of so-called colored cones.

The following conjecture in formulates in~\cite[Section~6]{AFKKZ-2}.

\begin{conjecture}
Any affine spherical variety without non-constant invertible regular functions is flexible.
\end{conjecture} 

The above results confirm this conjecture, but in general it remains open.

%%%%%%%%%%%%%%%%%
\subsection{Vector bundles} Let $\pi\colon E \to X$ be a reduced irreducible linear space over a flexible variety $X$, which is a vector bundle over $X^{\reg}$. Suppose there is an action of the group $\SAut(X)$ on $E$ such that the action of each $\GG_a$-subgroup in $\SAut(X)$ is algebraic and the morphism $\pi$ is equivariant. In~\cite[Corollary~4.5]{AFKKZ-1} it is shown that the total space $E$ is a flexible variety. In particular, the total space of the tangent bundle ~$TX$ and, more generally, of any tensor bundle
$E=(TX)^{\times a}\times(T^*X)^{\otimes b}$ is flexible.

\smallskip

In~\cite{Du} one can find an example of a smooth rational affine surface $S$ with a finite automorphism group such that the automorphism group of the cylinder
$S\times\AA^2$ acts infinitely transitively on the complement in this cylinder to a closed subset of codimension $\ge 2$. This example shows that there are rigid varieties such that the cylinders over them are generically flexible.

%%%%%%%%%%%%%%%%%%%%%%%
\subsection{Affine cones over projective varieties} Let $X$ be an affine cone over a projective variety $Y$ polarized by an ample divisor $H$.
By an affine cone we mean the spectrum of the homogeneous coordinate ring 
$$
\bigoplus_{n\ge 0} H^0(Y, nH).
$$
If the divisor~$H$ is very ample and we consider the image of the variety $Y$ under embedding into a projective space by means of a complete linear system~$|H|$, the variety $X$ coincides with the affine cone over the projective subvariety $Y$.

A natural question is the characterization of flexibility of the variety $X$ in terms of the pair $(Y,H)$. The study of this question began in the work of Perepechko~\cite{Pe}. It is based on the results of a series of papers by Kishimoto-Prokhorov-Zaidenberg. We briefly recall the relevant concepts and results. A detailed explanation can be found, for example, in a recent survey~\cite{CPPZ}.

An open subset $U$ of a variety $Y$ is called a \emph{cylinder} if $U\cong Z\times\AA^1$ for some smooth affine variety $Z$. A cylinder $U$ is called \emph{$H$-polar} if $U = Y\setminus\text{Supp}\,D$ for some effective $\QQ$-divisor $D$ linearly equivalent to $H$. In the works of Kishimoto-Prokhorov-Zaidenberg it is shown that an $H$-polar cylinder $U$ on $Y$ determines a $\GG_a$-action on the affine cone $X$ over $Y$.

We call a subset $W\subseteq Y$ \emph{invariant} with respect to a cylinder $U\cong Z\times\AA^1$ if $W\cap U = \pi^{-1}(\pi(W))$, where $\pi\colon U \to Z$ is the projection along $\AA^1$. In other words, a subset $W$ is invariant if every $\AA^1$-fiber of the cylinder is either contained in $W$ or does not intersect~$W$. A variety $Y$\emph is \emph{transversally covered} by cylinders $U_i$, $i=1,\ldots,s$ if $Y = \bigcup_i U_i$ and no proper subset $W\subseteq Y$ is invariant with respect to all elements of the covering $\{U_i\}$.

Theorem~2.5 in \cite{Pe} states that if for some very ample divisor $H$ on a normal projective variety $Y$ there is a transversal cover by $H$-polar cylinders, then the corresponding affine cone $X$ over $Y$ is flexible. This result allows to prove that any affine cone over a del Pezzo surface of degree $\ge 5$ is flexible. In~\cite{Pe} flexibility is also proved for some affine cones over del Pezzo surfaces of degree $4$, including the pluri-anticanonical cone. In~\cite{PW}, the result on flexibility is obtained for all affine cones over del Pezzo surfaces of degree~$4$.

On the other hand, it is known that pluri-canonical cones over del Pezzo surfaces of degrees $1$ and $2$ are rigid, that is, they do not admit nontrivial $\GG_a$-actions, see \cite[Corollary~1.8]{CPW} and \cite[Theorem~1.1]{KPZ-2}. The question of the rigidity of a pluri-canonical cone over del Pezzo surfaces of degree~$3$ has remained open for more than 15 years. The proof of this fact is obtained in~\cite{CPW}. It is all the more surprising that any very ample divisor on a del Pezzo surface of degree~$3$ that is not proportional to the anticanonical divisor defines an affine cone $X$ on which the group $\SAut(X)$ acts with an open orbit~\cite{Pe-2}. The presence of nontrivial $\GG_a$-actions on affine cones corresponding to the anticanonical divisor on del Pezzo surfaces with singularities is investigated in~\cite{CPW-2}.

\smallskip

In~\cite{MPS}, the following result is obtained.  

\begin{proposition} [{\cite[Theorem~1.4]{MPS}}] \label{propth}
Suppose that a normal projective variety $Y$ is covered by flexible affine charts $U_i$, ${i\in I}$ and we fix a very ample divisor $H$ on~$Y$. If each subset $Y\setminus U_i$ is the support of an effective $\QQ$-divisor $D_i $ that is linearly equivalent to the divisor~$H,$ then the affine cone $\text{AffCone}_HY$ over $Y$ is a flexible variety.
\end{proposition}

By the \emph{Segre-Veronese variety} $v_{s_1}(\PP^{d_1})\times\ldots\times v_{s_k}(\PP^{d_k})$ we mean the image of embedding of the direct product $\PP^{d_1}\times\ldots\times\PP^{d_k}$ of projective spaces, where each of the factors is embedded into a larger projective space by the Veronese map with the parameters $s_1,\ldots,s_k$, respectively, and the product is embedded by the Segre map. Further, for each projective subvariety $X\subseteq\PP^n$ we call the \emph{variety of secants}
(respectively, the \emph{variety of tangents}) of $X$ the Zariski closure of the union of secant (respectively, tangent) lines of the subvariety $X$ in~$\PP^n$.
In \cite[Theorem~2.20]{MPS} it is shown that the affine cone over the variety of secants to a Segre-Veronese variety is flexible. Moreover, if $s_1=\ldots=s_k=1$, then the affine cone over the variety of tangents of a Segre-Veronese variety turns out to be flexible.

Theorem~4.5 in~\cite{MPS} guarantees the flexibility of affine cones over some three-dimensional Fano varieties. In~\cite{PZ-2}, the flexibility of affine cones over four-dimensional Fano-Mukai varieties of genus~$10$ is proved.

%%%%%%%%%%%%%%%%%%%
\subsection{Universal torsors} The concept of a universal torsor originated in the works of Colliot-Th\'el\`ene and Sansuc \cite{CTS-1,CTS-2} on arithmetic algebraic geometry and was used to study rational points on algebraic varieties. In particular, universal torsors demonstrated their effectiveness for obtaining positive results on the Manin Conjecture that provides an asymptotic formula for the number of rational points of bounded height.

Let us recall briefly the definition. Let $X$ be a normal algebraic variety with a free finitely generated divisor class group $\Cl(X)$. Denote by $\WDiv(X)$ the group of Weil divisors on $X$ and fix a subgroup $K\subseteq\Div(X)$ that projects onto $\Cl(X)$ isomorphically. Consider the Cox sheaf $\RRR$ of the variety $X$ associated with the subgroup $K$: to each open subset  $U\subseteq X$ this sheaf put in correspondence the algebra of sections
$$
\bigoplus_{D\in K} H^0(U,D).
$$ 
The algebra of global sections of this sheaf is called the \emph{Cox ring} $R(X)$ of the variety $X$. It is easy to see that the objects defined here up to isomorphism do not depend on the choice of the subgroup~$K$.

The relative spectrum of the sheaf $\RRR$ determines the morphism
$$
q\colon \widehat{X}:=\Spec_X \RRR\to X.
$$
In the case when the variety $X$ is smooth, the map $q$ is a principal locally trivial bundle whose fiber is a torus~$H$ with the character group $\Cl(X)$. This mapping has certain universal properties and is called the \emph{universal torsor} over the variety~$X$. It is known that $\widehat{X}$ is a smooth quasi-affine variety.

If the Cox ring $R(X)$ is finitely generated, we can consider its spectrum ${\overline{X}:=\Spec R(X)}$. It is called the \emph{total coordinate space} of the variety $X$. This is an affine factorial variety, and the variety $\widehat{X}$ is embedded into $\overline{X}$ as an open subset whose complement does not contain divisors. For more information on this construction and its generalizations, see~\cite{Co} and~\cite[Chapter~I]{ADHL}.

\smallskip

Following the work~\cite{APS}, we say that a variety $X$ is \emph{$A$-covered} if there is a cover of $X$ by open subsets $U_i$ isomorphic to the affine space $\AA^n$. It is clear that an $A$-covered variety is a smooth rational variety with a free finitely generated divisor class group. In~\cite{APS} examples of wide classes of $A$-covered varieties are given.
In particular, these are all smooth complete toric varieties, all flag varieties, and, more generally, all smooth complete spherical varieties. Using the fact that the property to be $A$-covered is preserved under blowing up a point, it can be shown that all smooth complete rational surfaces are $A$-covered. In~\cite[Appendix]{APS} it is proved that all smooth complete rational varieties with an action of a torus of complexity one are $A$-covered. 

\begin{theorem} [{\cite[Theorem~3]{APS}}] \label{tut}
Let $X$ be an $A$-covered variety of dimension $\ge 2$ and $q\colon\widehat{X}\to X$ be a universal torsor over $X$. Then the group $\SAut(\widehat{X})$ acts on the variety $\widehat{X}$ infinitely transitively. 
\end{theorem}

The proof of this result is based on the fact that every open chart on $X$ isomorphic to~$\AA^n$ carries $n$ transversal structures of cylinders, each of which defines a $\GG_a$-action on the universal torsor.

\smallskip

Theorem~\ref{tut} allows to obtain many interesting examples of smooth flexible quasi-affine varieties. For instance, in this way one can construct a flexible quasi-affine variety with a non finitely generated algebra of regular functions~\cite[Section~5]{APS}.

\smallskip

Note that Theorem~\ref{tut} implies only that the group $\SAut(\overline{X})$ has a big open orbit on the total coordinate space $\overline{X}$~\cite[Theorem~4]{APS}. The question of the flexibility of the variety $\overline{X}$ remains open. Partial results are known in this direction: in~\cite[Theorem~5.4]{MPS} it is proved that the total coordinate space over a smooth del Pezzo surface is flexible, and \cite[Theorem~5.9]{MPS} guarantees the flexibility property for total coordinate spaces over complete smooth rational $T$-varieties of complexity one.

\smallskip

%%%%%%%%%%%%%%%%%%%%
\subsection{Gizatullin surfaces} This is what normal affine surfaces are called, which admit a completion by a chain of smooth rational curves. Gizatullin's Theorem~\cite[Theorems~2 and 3]{Gi} (see also~\cite{Du-1}) states that a normal affine surface $X$ other than
$$
\AA^1\times \big(\AA^1\setminus\{0\}\big)
$$
admits such a completion if and only if the group $\SAut(X)$ acts on $X$ with an open orbit; in this case the complement to the open orbit is finite. 

Gizatullin~\cite{Gi} conjectured that if the ground field $\KK$ has characteristic zero then the open orbit coincides with the set of smooth points, that is, $X$ is flexible. At that time, it was known that this is not the case in positive characteristic, since here the full automorphism group of a surface $X$ can have smooth fixed points~\cite{GD}. As we have seen above,  Gizatullin's conjecture is true for Gizatullin surfaces defined in $\AA^3$ by equations of the form $xy - f(z) = 0$; in this case we obtain a suspension over the line $\AA^1$. A wide class of surfaces for which Gizatullin's conjecture holds form the so-called Danilov-Gizatullin surfaces~\cite{Giz}.

However, in~\cite{Ko} Kovalenko constructed a counterexample to Gizatullin's conjecture over the field of complex numbers: in~\cite[Theorem~3.11]{Ko} classes of smooth Gizatullin surfaces $X$ are described for which the action of the group $\Aut(X)$ on $X$ is not transitive.

%%%%%%%%%%%%%%%%%%
\subsection{Calogero-Moser spaces} Recall that the \emph{Calogero-Moser space} is the space of the categorical quotient
$$
C_n=\big\{(X,Y)\in\Mat_n(\CC)\times\Mat_n(\CC) \, :\, \rk([X,Y]+\Id)=1\big\} /\!/\GL_n(\CC),
$$
where the group $\GL_n(\CC)$ acts on pairs of matrices by conjugation:
$$
g(X,Y)=(gXg^{-1},gYg^{-1}).
$$
Such spaces play an important role in representation theory and other branches of mathematics. It is known that $C_n$ is a smooth irreducible affine variety of dimension $2n$. It carries a hyperk\"ahler structure and is an example of a Nakajima quiver variety. This variety arises as a partial completion of the integrable Calogero-Moser system. These and other facts with links to the original works can be found in~\cite[Section~1]{BEE}.

Let us consider on the space of pairs of matrices the mappings of the form
$$
(X,Y)\mapsto (X+p(Y),Y) \quad \text{and} \quad (X,Y)\mapsto (X,Y+q(X)),
$$
where $p(x)$ and $q(x)$ are arbitrary polynomials in one variable. It is easy to see that these mappings induce automorphisms of the space $C_n$. Let us denote by $G$ the group of automorphisms of the space $C_n$ generated by these automorphisms. In~\cite[Theorem~1]{BEE} it is proved that the action of~$G$ on~$C_n$ is $2$-transitive. In the same paper, the authors consider the diagonal action of the group $G$ on the direct product $C_{n_1}\times\ldots\times C_{n_k}$ for any pairwise distinct positive integers $n_1,\ldots,n_k$. We call such an action \emph{collectively infinitely transitive} if for any positive integers $m_1,\ldots,m_k$ and any sets of pairwise distinct $m_1$ points on $C_{n_1}$, pairwise distinct $m_2$ points on $C_{n_2}$ and so on, there is an element of the group $G$ that translates these sets to any other predefined sets of pairwise distinct points of the same cardinalities.

In~\cite{BEE}, the authors conjectured that the action of the group $G$ on the space $C_n$ is infinitely transitive, and the diagonal action described above has the property of collective infinite transitivity. These conjectures are proved in~\cite[Theorem~3]{Ku}. Independently, in~\cite{An-1} the author checks the flexibility of the variety $C_n$, which implies infinite transitivity of the group $\SAut(C_n)$ containing $G$ as a proper subgroup.

%%%%%%%%%%%%%%%%%%%%%%%%%%
\section{Infinite transitivity and finite generation}
\label{sec-4}

In this section we discuss a recently discovered nontrivial effect. It would seem that the infinite transitivity of the action of the group of special automorphisms on the set of smooth points of a flexible affine variety is caused by the fact that we include saturated sets of $\GG_a$-subgroups, in particular, all replicas of $\GG_a$-subgroups, in the acting group. In this context, the following conjecture, formulated in ~\cite{AKZ-2}, may look unexpected.

\begin{conjecture} \label{concon} 
Let $X$ be an affine variety of dimension $\ge 2,$ where the set of flexible points forms a nonempty open subset~$O$. Then there is such a finite collection $\{H_1,\ldots,H_k\}$ of $\GG_a$-subgroups in the automorphism group $\Aut(X)$ that the subgroup $G=\langle H_1,\ldots,H_k\rangle$ generated by these subgroups as an abstract group acts on $O$ infinitely transitively.
\end{conjecture}

%%%%%%%%%%%%%%%%%%%%%%%%%%%%%%%
\subsection{The case of an affine toric variety} In~\cite{AKZ-2} and subsequent publications, a number of confirmations of Conjecture~\ref{concon} were obtained. The most general result in this direction is the following theorem.

\begin{theorem}[{\cite[Theorem~5.20]{AKZ-2}}] \label{tcod2}
Let $X$ be a non-degenerate affine toric variety of dimension $\ge 2$. Assume that $X$ is smooth in codimension two. Then there is a finite collection $H_1,\ldots,H_k$ of root subgroups on $X$ such that the group $G=\langle H_1,\ldots,H_k\rangle$ acts on the smooth locus $X^{\reg}$ infinitely transitively.
\end{theorem}

Below we describe a scheme of the proof of this and similar results. The proof is divided into four stages. At the first stage, a description of certain countable sets of $\GG_a$-subgroups of the group $\Aut(X)$ is obtained for which the group $G$ generated by them acts on $X$ with an open orbit and the action of $G$ on the open orbit is infinitely transitive, see ~\cite[Theorem~2.2, Corollary~2.8]{AKZ-2}.

At the second stage we take the closure $\overline{G}$ of the subgroup $G$ in the automorphism group $\Aut(X)$ considered as an ind-group. Let us discuss this stage in more detail. Following~\cite[Proposition~2.1]{KPZ}, we fix an embedding $X\hookrightarrow\AA^n$ and consider the degree function on the algebra $\KK[X]$ induced by the embedding. For each automorphism $\alpha\in\Aut(X)$, we define $\deg(\alpha)$ as a maximum of degrees of homogeneous components of the images of elements of a fixed generating system under the action of the automorphism $\alpha$.

Then we have $\Aut(X)=\varinjlim\Sigma_s$, where 
\begin{itemize}
\item for $s\ge 1$ the subset $\Sigma_s:=\{\alpha\in\Aut(X)\,|\,\deg(\alpha), \deg(\alpha^{-1})\le s\}$ is a closed subvariety of the variety $\Sigma_{s+1}$;
\item for $r,s\ge 1$ the composition defines a morphism $\Sigma_r\times\Sigma_s\to\Sigma_{rs}$;
\item the passage to the inverse element defines an automorphism of $\Sigma_s$.
\end{itemize}

The Zariski closure of a subset $F\subseteq\Aut(X)$ is defined as
$$
\overline{F}= \varinjlim\overline{(F\cap\Sigma_s)},
$$ 
where the dash means the Zariski closure in the variety $\Sigma_s$. The Zariski closure of a subset $F$ is a closed ind-subvariety in the ind-variety $\Aut(X)$. In these terms, an algebraic subgroup in $\Aut(X)$ is a subgroup that is a closed subvariety in some~$\Sigma_s$.

It is not difficult to verify that the closure $\overline{G}$ of a subgroup $G\subseteq\Aut(X)$ is a closed ind-subgroup in $\Aut(X)$. Further, if $\rho\colon\mathbb{A}^1\to\Aut(X)$ is such a morphism that $\rho(t)\in G$ for $t\neq 0$, then we have $\rho(0)\in\overline{G}$.

\begin{lemma}[{\cite[Lemma~3.2]{AKZ-2}}] 
Any $G$-invariant closed subset $Y\subseteq X$ is $\overline{G}$-invariant. If the group $G$ acts on $X$ with an open orbit $O_G,$ then $O_G$ coincides with the open orbit $O_{\overline{G}}$ of the group $\overline{G}$.
\end{lemma}

These observations in the case of algebraically generated subgroups allow us to prove the following.

\begin{proposition}[{\cite[Proposition~3.4]{AKZ-2}}] 
Let $G$ be an algebraically generated subgroup of the group $\Aut(X)$. Then
\begin{itemize}
\item[{\rm (a)}] the orbits of the groups $G$ and $\overline{G}$ on $X$ coincide. In particular, if $\overline{G}$ acts with an open orbit $O_{\overline{G}},$ then the same is true for $G$ and $O_{G}=O_{\overline{G}};$
\item[{\rm (b)}] if $\overline{G}$ acts $m$-transitively on $O_{\overline{G}},$ then the same is true for~$G;$
\item[{\rm (c)}] if $\overline{G}$ acts infinitely transitively on $O_{\overline{G}},$ then the same is true for~$G$.
\end{itemize}
\end{proposition} 

Thus, the results of the second stage allow us to prove the infinite transitivity not for the action of the group $G$, but for the action of its closure.

At the third stage we guarantee that $\overline{G}$ contains extra root subgroups with respect to~$G$. This requires a certain degeneration technique. Namely, consider the locally nilpotent derivation $\partial$ of the algebra $\KK[X]$ corresponding to a $\GG_a$-subgroup of the group $G$, and define its Newton polyhedron $N(\partial)$ as the convex hull of weights of homogeneous components of the derivation $\partial$. Vertices of this polyhedron correspond to root subgroups, and these subgroups are contained in the group $\overline{G}$, see ~\cite[Corollary~4.17]{AKZ-2}.

In order to find a non-root subgroup with which to start this process, we conjugate one root subgroup in $G$ by another root subgroup that does not centralize the first one. The Newton polyhedron $N(\partial)$ for locally nilpotent derivation corresponding to the result of such conjugation turns out to be a segment whose ends correspond to new root subgroups in $\overline{G}$. This segment can be calculated explicitly using the Baker-Campbell-Hausdorff formula, see ~\cite[Corollary~4.14]{AKZ-2}. On this path, we come to the next technical result.

Let $X$ be an affine toric variety of dimension $\ge 2$. Below we use the terminology and notation introduced in Section~\ref{3.2}. Consider the Demazure roots $e_1,e_2$ of the cone of the variety $X$ corresponding to primitive vectors $p_1,p_2$ on the edges of the cone, and the corresponding root subgroups $H_{e_1},H_{e_2}$ in the group $\Aut(X)$. Let $d=\langle p_1, e_2\rangle+1$ and assume that $e_2+de_1$ is a Demazure root associated with the vector $p_1$, or, equivalently, that $\langle p_2, e_1\rangle>0$.

\begin{lemma}[{\cite[Lemma~4.18]{AKZ-2}}] \label{techkey} 
The root subgroup $H_{e_1+de_2}$ lies in the subgroup $\overline{\langle H_{e_1} , H_{e_2}\rangle}$.
\end{lemma} 

Finally, at the fourth stage~\cite[Section~5.3]{AKZ-2}, we find some finite collection of root subgroups $H_1,\ldots,H_k$ in the group $\Aut(X)$ and generate a subgroup $G$ by them. Repeatedly applying Lemma~\ref{techkey}, we find a countable set of root subgroups in the group $\overline{G}$, which satisfies the conditions of the first stage. This completes the proof of Theorem~\ref{tcod2}.

\begin{remark}
The condition of smoothness in codimension two is essentially used at the fourth stage of the proof of Theorem~\ref{tcod2} when constructing a collection of subgroups $H_1,\ldots,H_k$. Unfortunately, we cannot find a way to construct the required collection that would not use this condition.
\end{remark}

Let us formulate a generalization of Lemma~\ref{techkey} as a conjecture.

\begin{conjecture} [{\cite[Conjecture~5.23]{AKZ-2}}] 
Let $X$ be an irreducible affine variety. Consider the group $G=\langle H_1,\ldots,H_k\rangle$ generated by a finite collection of $\GG_a$-subgroups $H_i=\exp(\KK\partial_i)$ in $\Aut(X),$ where $\partial_i$ is a locally nilpotent derivation of the algebra $\KK[X]$. Then a $\GG_a$-subgroup $H=\exp(\KK\partial)$ lies in $\overline{G}$ if and only if the locally nilpotent derivation $\partial$ belongs to the Lie algebra generated by $\partial_1,\ldots,\partial_k$. 
\end{conjecture}

Currently, this conjecture has not been proved.

%%%%%%%%%%%%%%%%%%%%%%%%%%
\subsection{The case of an affine space} The scheme of the proof of Theorem~\ref{tcod2} described above allows us to obtain more explicit results in the case $X=\AA^n$.

Let us start with the case $n\ge 3$. We fix a positive integer $l$ and consider a $\GG_a$-subgroup $F_l$ of transformations of the space $\AA^n$, which transforms coordinates according to the formula $(x_1+tx_2^l,x_2,\ldots,x_n)$, $t\in\KK$. Let $S_n$ be the symmetric group acting on $\AA^n$ by permutations of coordinates. Denote by $G_l$ the subgroup in $\Aut(\AA^n)$ generated by the subgroups $F_l$ and~$S_n$.

\begin{theorem}[{\cite[Theorem~5.2]{AKZ-2}}]
The group $G_2$ acts on the set $\AA^n\setminus\{0\}$ infinitely transitively.
\end{theorem}

Note that for $l\ne 2$ the action of the group $G_l$ on $\AA^n\setminus\{0\}$ is not even $2$-transitive. Indeed, for $l=1$, this action is linear and, therefore, preserves the property of a pair of vectors to be collinear. For $l>2$, we fix a primitive root of unity $\omega$ of degree $l-1$ and consider the set of pairs of points on $\AA^n\setminus\{0\}$ of the form $(P,\omega P)$. It is easy to check that transformations from the group $G_l$ send pairs of this form to pairs of the same form.

\smallskip

Now we come to the case $n=2$. Here we define two one-parameter subgroups of the automorphism group: the subgroup $F_l$ acts according to the formula $(x_1+t_1x_2^l,x_2)$, and the subgroup $R_s$ acts as $(x_1,x_2+t_2x_1^s)$, $t_1,t_2\in\KK$. Let $G_{l,s}$ be the subgroup in $\Aut(\AA^2)$ generated by the subgroups $F_l$ and $R_s$.

With $ls\!\ne\! 2$ the action of the group $G_{l,s}$ on $\AA^2$ is not even $2$-transitive. Indeed, for $l=0$ or $s=0$, the action preserves the differences of the corresponding coordinates. If $l=s=1$, then the action is linear. Finally, if $ls>2$, we fix the primitive root of unity $\omega$ of degree $ls-1$ and consider the set of pairs of points on $\AA^2\setminus\{0\}$ of the form $((x_1,x_2), (\omega x_1,\omega^sx_2))$. A direct check shows that transformations from the group $G_{l,s}$ send such pairs to pairs of the same type.

Thus, it remains to consider the case $ls=2$, or, up to permutation, $l=2$ and $s=1$. So far we have not been able to prove the infinite transitivity of the action of the group $G_{2,1}$ on the set $\AA^2\setminus\{0\}$ within the framework of the plan described above. After we formulated this property in a conjectural form, its proof was obtained using a different technique.

\begin{theorem}[{\cite[Corollary~21]{LPS}}] \label{twoinf} 
The action of the group $G_{2,1}$ on the set $\AA^2\setminus\{0\}$ is infinitely transitive. 
\end{theorem} 

The main idea of~\cite{LPS} is to study relationships between two structures on the group $\Aut(\AA^2)$, the structure of the ind-group and the structure of the amalgamated free product. Using a technique related to the so-called Polydegree Conjecture, the authors show that the closure of the subgroup $G_{2,1}$ contains all subgroups
$F_l$ for $l\ge 2$, see ~\cite[Theorem~20]{LPS}. This result implies Theorem~\ref{twoinf}. It would be interesting to find a more elementary proof of this theorem.~\footnote{A new proof and generalizations of this result are given in a recent preprint Chistopolskaya~A. and Taroyan~G., Infinite transitivity for automorphism groups of the affine plane, arXiv:2202.02214}

\smallskip 

Let us return to the case of an affine space $\AA^n$ of arbitrary dimension. We denote by $\Aff_n$ the group of affine transformations of the space $\AA^n$.
We also consider the group of tame automorphisms $\Tame_n$, see Section~\ref{1.3}.

\begin{definition}
An automorphism $h\in\Aut(\AA^n)$ is called \emph{co-tame} if $\langle \Aff_n, h\rangle=\Tame_n$, and \emph{topologically co-tame} if $\overline{\langle\Aff_n, h\rangle}=\Tame_n$.
\end{definition} 

The Edo Theorem, generalizing the results of Bodnarchuk and Furter, claims that for $n\ge 2$ any automorphism $h\in\Aut(\AA^n)\setminus\Aff_n$ is topologically co-tame. This result leads to the following theorem.

\begin{theorem}[{\cite[Theorem~5.12]{AKZ-2}}]
For $n\ge 2$ and any automorphism $h\in\Aut(\AA^n)\setminus\Aff_n$ the group $G=\langle\Aff_n, h\rangle$ acts on $\AA^n$ infinitely transitively.
\end{theorem}

Bodnarchuk, Levis and Edo showed that for $n\ge 3$ any triangular non-affine automorphism of the space $\AA^n$ is co-tame, whereas for $n=2$ a triangular automorphism cannot be  co-tame. A detailed presentation of these results and links to original works can be found in~\cite[Section~5.2]{AKZ-2}. It should be noted that back in 1995 in the work~\cite{Bod} Bodnarchuk related these results to the infinite transitivity of the action of subgroups of the automorphism group of an affine space.

\smallskip

It is natural to ask for what minimum value $k$ there is a group $G=\langle H_1,\ldots,H_k\rangle$ generated by $\GG_a$-subgroups that acts on the space $\AA^n$ infinitely transitively. It turns out that the value $k=3$ is suitable. For example, in the case $n=2$ one may consider the subgroup generated by $G_{2,1}$ and an arbitrary one-dimensional subgroup of parallel translations. Also for any $n\ge 2$ we construct the required triple of subgroups $H_1, H_2, H_3$ in ~\cite[Theorem~5.17]{AKZ-2}. This construction is based on a result of Chistopolskaya~\cite{Chi}, which states that for an arbitrary nonzero nilpotent matrix $x$, there is such a nilpotent matrix $y$ that $x$ and $y$ generate the Lie algebra $\sl_n$. Another triple of subgroups having the same property is constructed independently in~\cite[Theorem~10]{An}.

The question of whether the value $k=2$ is achievable for any $n\ge 2$ remains open.

\begin{remark}
In~\cite[Proposition~1.14]{FKZ} it is shown that for any flexible variety $X$ in the group $\SAut(X)$ there are two $\GG_a$-subgroups $H_1, H_2$ such that
the group $G$ generated by the subgroups $H_1, H_2$ and all their replicas acts on $X$ with an open orbit $O$, and the action of $G$ on $O$ is infinitely transitive.
\end{remark}

%%%%%%%%%%%%%%%%%%%%%%%%%%%%%%%
\subsection{Groups generated by root subgroups and the Tits alternative} At the end of this section, we outline briefly the results of the work~\cite{AZ}. The starting point was Demailly's question about whether it is possible to characterize the infinite transitivity property of the action of the group $G=\langle H_1,\ldots,H_k\rangle$ on its open orbit in terms of the growth of finitely generated subgroups of the group $G$.

In this context, it is natural to mention the classical result of Tits~\cite{Ti}: any finitely generated subgroup in a linear algebraic group over an arbitrary field contains either a solvable subgroup of finite index (the virtual solvability property) or a non-commutative free subgroup. If the ground field has characteristic zero, this alternative holds for any, not necessarily finitely generated subgroup.

We say that a group $G$ satisfies the \emph{Tits alternative} if any its subgroup is either virtually solvable or contains a non-commutative free subgroup. It is known that the Tits alternative holds for the group $\Aut(\AA^2)$ and, more generally, for the group of birational automorphisms of a compact complex K\"ahler surface, in particular, for the two-dimensional Cremona group. The alternative also holds for the automorphism group of a three-dimensional smooth affine quadric and for a number of other automorphism groups; references to the original works can be found, for example, in the Introduction to~\cite{AZ}.

For groups generated by root subgroups, there is a weaker result.

\begin{theorem}[{\cite[Theorem~1.1]{AZ}}] \label{tta} 
Let $X$ be an affine toric variety and $G=\langle H_1,\ldots,H_k\rangle$ be a subgroup of $\Aut(X)$ generated by a finite set of root subgroups. Then either $G$ is a unipotent linear algebraic group, or $G$ contains a non-commutative free subgroup.
\end{theorem} 

The question whether the group $G$ from Theorem~\ref{tta} satisfies the Tits alternative is open. At the same time, Theorem~\ref{tta} shows that the property of infinite transitivity of an action of a group on its open orbit cannot be characterized only in terms of maximal growth of subgroups: any non-unipotent group $G=\langle H_1,\ldots,H_k\rangle$ contains a non-commutative free subgroup, that is, a subgroup of the (maximal possible) exponential growth. Nevertheless, the following conjecture remains open.

\begin{conjecture}[{\cite[Conjecture~1]{AZ}}]
Let $X$ be an affine variety of dimension $\ge 2$. Consider a subgroup $G=\langle H_1,\ldots,H_k\rangle$ in $\Aut(X)$ generated by a finite collection of $\GG_a$-subgroups. Suppose that $G$ acts $2$-transitively on some of its orbits. Then $G$ contains a non-commutative free subgroup.
\end{conjecture}

In the proof of Theorem~\ref{tta}, we start with the case $k=2$. If the subgroups $H_1$ and $H_2$ commute elementwise, the group $G$ is commutative and unipotent. Therefore, we can assume that the Demazure roots $e_1, e_2$ corresponding to the subgroups $H_1, H_2$ are associated with primitive vectors $p_1, p_2$ on different rays of the cone of the affine toric variety $X$. In this case, everything is determined by the values of the pairings $\langle p_1, e_2\rangle$ and $\langle p_2, e_1\rangle$: if at least one of these pairings is zero, the group $G$ turns out to be unipotent, and in the case when both pairings are positive, we prove the existence of a non-commutative free subgroup in the group $G$. For example, if both pairings are $\ge 2$, the group~$G$ turns out to be a free product of the subgroups $H_1, H_2$ and any two nonunit elements $h_1\in H_1$ and $h_2\in H_2$ generate a free subgroup; see details in~\cite[Section~3]{AZ}.

In the case $k\ge 3$, we can assume without loss of generality that for any pair $e_i,e_j$ of Demazure roots corresponding to root subgroups of the group $G$, at least one of the pairings, $\langle p_i, e_j\rangle$ or $\langle p_j, e_i\rangle$, is zero. In this situation, we prove that the group $G$ is unipotent~\cite[Proposition~4.8]{AZ}. This result is based on the study of the Lie algebra generated by a finite collection of homogeneous locally nilpotent derivations.

The methods used here are close to the methods of~\cite{ALS}, where a criterion of the finite dimensionality of a Lie algebra generated by a finite collection of homogeneous locally nilpotent derivations is obtained, and the structure of such finite-dimensional Lie algebras is described. In particular, the results of~\cite{ALS} generalize the well-known Demazure Theorem~\cite{De}: the Lie algebras arising here have type~A, that is, the semisimple parts of such Lie algebras are direct sums of Lie algebras~$\sl$.

%%%%%%%%%%%%%%%%%%%%%%%%%%%%%%%%%%%%%
\section{Appendix A. Infinitely transitive actions in complex analysis}
\label{ap-b}

In this section we consider complex-analytic versions of flexibility and infinite transitivity of the action of the automorphism group, and also discuss connections of these concepts and the constructions studied above with the Andersen-Lempert theory and the concept of Gromov sprays.

\begin{definition}[{\cite[Section~1.1.B]{Gr}}]
Let us consider a complex manifold~$X$. 
\begin{itemize}
\item[(i)]
A \emph{dominating spray} on $X$ is a holomorphic vector bundle $\rho\colon E\to X$ with a holomorphic map $s\colon E\to X$ such that its restriction $s$ to the zero section $Z$ is the identical map, and for each point ${x\in Z\cong X}$ the tangent mapping $d_x s$ sends the fiber $E_x\!:=\!\rho^{-1}(x)$ (viewed as a subspace of the tangent space $T_xE$) to the tangent space $T_xX$ surjectively.
\item[(ii)]
Let $h\colon X\to B$ be a surjective submersion of complex manifolds. We say that $h$ admits a \emph{fiber dominating spray} if there exists a holomorphic vector bundle $E$ over $X$ and a holomorphic map $s\colon E\to X$ such that the restriction of $s$ on each fiber $h^{-1}(b), b\in B$ yields a dominating spray on this fiber.
\end{itemize}
\end{definition}

In these terms, the well-known Oka-Grauert-Gromov principle can be stated as follows.

\begin{theorem}[{\cite[Section~4.5]{Gr}}] 
Consider a surjective submersion
$$
{h\colon X\to B}
$$
of Stein manifolds. If $h$ admits a fiber dominating spray, then
\begin{itemize}
\item[(a)]
any continuous section of the mapping $h$ is homotopic to a holomorphic section;
\item[(b)]
if two holomorphic sections are homotopic in the class of continuous maps, then they are also homotopic holomorphically.
\end{itemize}
\end{theorem}

The following result shows that this principle is applicable to smooth affine algebraic $G$-bundles with flexible fibers.

\begin{proposition}[{\cite[Proposition~A.3]{AFKKZ-1}}] \mbox{}\newline
\vspace{-5mm}
\begin{itemize}
\item[(a)]
Every smooth flexible affine complex algebraic variety $X$ admits a dominating spray.
\item[(b)]
Let $h\colon X\to B$ be a surjective submersion of smooth affine complex algebraic varieties. Assume that there exists an algebraically generated subgroup $G\subseteq\Aut(X),$ whose orbits coincide with the fibers of $h$. Then the mapping ${X\to B}$ admits a fiber dominating spray.
\end{itemize}
\end{proposition} 

\begin{proof}
It is clear that assertion~(a) follows from assertion~(b). To prove (b), we consider a sequence of algebraic subgroups 
$$
\OH=(H_1,\ldots,H_s)
$$
in the group $G$ such that the tangent space to the orbit $Gx$ at each point $x\in X$ is generated by tangent vectors to the orbits $H_ix$, ${i=1,\ldots,s}$. Let $\text{exp}\colon T_1(H_i)\to H_i$ be the exponential map. Consider the trivial vector bundle
$$
E=\prod_{i=1}^s T_1(H_i)\times X
$$
over $X$ and the morphism
$$
s\colon E\to X, \quad \big((h_1,\ldots,h_s),x\big)\mapsto \exp(h_1)\ldots\exp(h_s)x. 
$$
It is not difficult to see that this is a fiber dominating spray. 
\end{proof}

Now we come to the definition of holomorphic flexibility. Recall that a holomorphic vector field on a complex manifold $X$ is called \emph{completely integrable} if
its phase flow determines a holomorphic action of the group $(\CC,+)$ on $H$.

\begin{definition}[{\cite[Definition~A.4]{AFKKZ-1}}]
\mbox{}\newline
\vspace{-5mm}
\begin{itemize}
\item[(i)]
A Stein manifold $\!X\!$ is called \emph{holomorphically flexible} if completely integrable holomorphic vector fields on $X$ generate the tangent space $T_xX$ at every smooth point on $X$.
\item[(ii)]
Consider a holomorphic submersion $h\colon X\to B$ of Stein manifolds. We say that $X$ is \emph{holomorphically flexible over $B$} if completely integrable relative holomorphic vector fields on $X$ generate the relative tangent bundle for $X\to B$ at each point $x\in X$. In this case, each fiber $h^{-1}(b), b\in B$ is a holomorphically flexible Stein manifold.
\end{itemize}
\end{definition}

For example, the vector field $\delta=z\frac{d}{dz}$ on $X=\CC\setminus\{0\}$ is completely integrable. However, the derivation $\delta$ is not locally nilpotent. The manifold $\CC\setminus\{0\}$ is not flexible in the sense of the main part of this work, but it is holomorphically flexible.

In the context of holomorphic actions, it also can be shown that if $X$ is a connected Stein manifold that is holomorphically flexible over a Stein manifold $B$, then the relative tangent bundle to $X$ over $B$ is generated by a finite number of completely integrable relative holomorphic vector fields on $X$ \cite[Lemma~A.6]{AFKKZ-1}. Repeating the reasoning given above, we show that if a Stein manifold $X$ is holomorphically flexible over a Stein manifold $B$, then the map $h\colon X\to B$ admits a fiber dominating spray~\cite[Corollary~A.7]{AFKKZ-1}. In particular, the Oka-Grauert-Gromov principle applies to the map $h\colon X\to B$.

It should be noted that the question of whether the group of holomorphic automorphisms acts on a flexible connected Stein manifold $X$ of dimension $\ge 2$ infinitely transitively remains open~\cite[Problem~A.8]{AFKKZ-1}. In \cite{AFKKZ-1} it is shown that such an action is transitive: it follows from the implicit function theorem that all orbits of the group of holomorphic automorphisms are open in $X$ in the standard Hausdorff topology, and so there is only one orbit.

At the same time, the infinite transitivity of an action can be proved under stronger restriction. To formulate the corresponding result, we need to introduce several concepts from the Andersen-Lempert theory; see, for example, \cite{KK,Va}.

\begin{definition}\mbox{}\newline
\vspace{-6mm}
\begin{itemize}
\item[(i)]
A complex manifold $X$ has the \emph{density property} if the Lie algebra generated by all completely integrable holomorphic vector fields is dense in the Lie algebra of all holomorphic vector fields on $X$ in the compact-open topology.
\item[(ii)]
An affine algebraic variety has the \emph{algebraic density property} if the Lie algebra generated by all completely integrable algebraic vector fields on $X$ coincides with the Lie algebra of all algebraic vector fields on~$X$.
\end{itemize}
\end{definition}

\begin{theorem}[{\cite{KK,Va}}] 
If a connected Stein manifold $X$ of dimension ${\ge 2}$ has the density property, then the group of holomorphic automorphisms of the manifold $X$ acts on $X$ infinitely transitively. Moreover, for any discrete subset $Z\subseteq X$ and any Stein space $Y$ of positive dimension that admits a proper embedding into $X,$ there is another proper embedding $\varphi\colon Y\hookrightarrow X$ such that $Z\subseteq\varphi(Y)$.
\end{theorem}

A similar result can be proved for holomorphic automorphisms and holomorphic vector fields preserving a volume form on a Stein manifold, see ~\cite[Corollary~2.1 and Remark~2.2]{KK} and~\cite[Theorem~A.12]{AFKKZ-1}.
 
The study of relationships between the flexibility of varieties and different versions of the concept of Gromov sprays has been conducted actively in recent years. Recall that a complex manifold $X$ is called \emph{elliptic} (respectively, \emph{subelliptic}) if it admits a dominating holomorphic spray (respectively, a dominating family of holomorphic sprays; see \cite[Definition~1.1]{KKT} for a precise definition).  

We say that a smooth algebraic variety $X$ is \emph{locally stably flexible} if $X$ is a union $\bigcup_i X_i$ of a finite number of Zariski-open quasi-affine subsets $X_i$, and for some positive integer~$m$ the varieties $X_i\times\CC^m$ are flexible for all $i$. In~\cite{KKT} it is shown that the blow-up of a locally stably flexible variety in a smooth algebraic subvariety is a subelliptic manifold. This result is obtained in~\cite{KKT} as a consequence of a more general result concerning so-called $k$-flexible varieties. 

Let us fix a positive integer $k$. A flexible quasi-affine variety $X$ is called \emph{$k$-flexible} if there exists such a morphism $\varphi\colon X\to Q$ to a normal affine variety $Q$ that for some Zariski-open subset $Q_0$ in $Q$ the preimage $\varphi^{-1}(Q_0)$ is isomorphic to the direct product $Q_0\times\CC^k$ over $Q_0$. Theorem~6.1 from ~\cite{KKT} states that if $X$ is a locally $k$-flexible variety for some $k\ge 2$ and $Z$ is a closed subvariety in $X$ of codimension at least $k$, then the result of the blow up of $X$ along~$Z$ is a subelliptic manifold.

In particular, varieties studied in~\cite{KKT} are Oka manifolds. Recall that a complex manifold $X$ is called an \emph{Oka manifold} if, for any positive integer $n$, any holomorphic map from a neighborhood of a compact convex subset $K$ in the space $\CC^n$ to $X$ is a uniform limit on $K$ of holomorphic maps from $\CC^n$ to $X$. For more information on Oka manifolds and the connection of this theory with the concept of flexibility one can see, for example, the work~\cite{Kut} and the recent survey~\cite{Fo}. The known facts on flexibility of toric varieties allow to apply the Oka theory and to obtain interesting results on extension, approximation and interpolation of maps and other geometric consequences both for toric varieties and for classes close to toric varieties, see, for example, \cite{LT}.

An exceptionally interesting object for research is the action of the group of holomorphic automorphisms of a Stein manifold on countable sets of points of such a manifold, see ~\cite{RR,AU,Wi-1,Wi-2} and references therein. In this situation, the orbits of the group of holomorphic automorphisms on the set of infinite discrete sequences form continuum families. Among such sets, there are those on which any permutation is obtained by restricting a holomorphic automorphism of the ambient space. Such sets are called \emph{tame}. In the case when the ambient space is the space $\CC^n$, the property of a set to be tame is equivalent to the fact that by a suitable holomorphic automorphism of $\CC^n$ this set can be sent to a standard position, for example, to the set of points with positive integer coordinates on the first coordinate axis. Every tame set in $\CC^n$ is \emph{disposable} in the sense that for such a set $E$ there exists a biholomorphic map $\CC^n\to \CC^n\setminus E$. On the other hand, there are examples of \emph{rigid} sets, that is, sets that are preserved only by the identical holomorphic automorphism.   

%%%%%%%%%%%%%%%%%%%%%%%%%%%%%%%%
\section{Appendix B. Infinitely transitive actions of finitely generated groups}
\label{ap-b}

Let us start with elementary properties of multiply transitive actions of abstract groups on arbitrary sets.

\begin{lemma} \label{lem-com}
A commutative group $G$ cannot act $2$-transitively on a set $X$ with at least $3$ elements.
\end{lemma}

\begin{proof}
Let $G$ act $2$-transitively on $X$. Let us choose three distinct elements $x_1,x_2,x_3\in X$. An element of $G$ sending the pair $(x_1,x_2)$ to the pair $(x_1,x_3)$ is contained in the stabilizer of the point $x_1$. From commutativity of the group it follows that the stabilizer of the point $x_1$ acts identically on the orbit of the point $x_1$, and hence on the entire set $X$, a contradiction.
\end{proof}

In particular, no cyclic group can act $2$-transitively. Now, following~\cite{MD}, we construct an example of an infinitely transitive action of a $2$-generated group on the set of integers $\ZZ$.

We say that a permutation $g$ of an infinite set $X$ is an \emph{infinite cycle} if the cyclic group $\langle g\rangle$ acts transitively on $X$. Consider two transformations on the set ~$\ZZ$: let $h\colon\ZZ\to\ZZ$ send $i$ to $i+1$, and $f$ be an infinite cycle on the set of positive integers $\NN$ that leaves all other integers fixed.

\begin{proposition}[{\cite[Lemma~1]{MD}}] \label{pr-2-bes}
The group $G:=\langle f,h\rangle$ acts on the set $\ZZ$ infinitely transitively.
\end{proposition}

Let us start with an auxiliary statement. For each non-negative integer $n$, we consider the subgroup $G_n:=\langle f, h^{-1}fh,\ldots,h^{-n}fh^n\rangle$ in $G$. It is easy to see that this subgroup leaves all integers not exceeding $-n$ fixed, and, therefore, it acts on the set $X_n$ consisting of all integers greater than $-n$.

\begin{lemma} \label{pollem}
The group $G_n$ acts on the set $X_n$ $(n+1)$-transitively.
\end{lemma}

\begin{proof}
We use the induction on $n$. For $n=0$, the subgroup $G_0$ coincides with $\langle f\rangle$. This group acts transitively on $X_0=\NN$ by construction.

Assume that the assertion is proved for all $i$ less than some $n$. Note that
$$
(h^{-n}fh^n)(-n+1)=f(1)-n\ne-n+1.
$$
Taking into account the inductive hypothesis, we conclude that the action of $G_n$ on $X_n$ is transitive. The stabilizer of the point $-n+1$ contains the subgroup $G_{n-1}$, which acts $n$-transitively on ${X_{n-1}=X_n\setminus\{-n+1\}}$. This proves the lemma.
\end{proof}

\begin{proof}[Proof of Proposition~\ref{pr-2-bes}]
Let us fix a positive integer $m$ and consider two sets of pairwise distinct integers of $m$ elements each. We choose $n\ge m-1$ so that any element from these sets
is greater than $-n$. By Lemma~\ref{pollem}, the first set can be sent to the second one by an element of the subgroup~$G_n$.
\end{proof} 

It is shown in \cite[Theorem~1]{MD} that by a suitable choice of the infinite cycle $h$ one can ensure that $G$ is isomorphic to the free group with two generators $f$ and $h$. It is well known that a free group with two generators contains a normal subgroup, which is a free group with $n$ generators for any positive integer $n$, as well as a normal subgroup, which is a free group with a countable number of generators, see, for example, \cite[Section~1]{MD}. Thus, Theorem~\ref{norgr} below allows us to construct an effective infinitely transitive action of a free group with any finite or countable number of generators on the set $\ZZ$.

Let us prove some more general statements.

\begin{lemma} \label{prtr}
Suppose that a group $G$ acts on a set $X$ with at least $3$ elements effectively and $2$-transitively. Let $N$ be a nontrivial normal subgroup in $G$. Then the action of $N$ on $X$ is transitive. In particular, the center $Z(G)$ of the group $G$ is trivial.
\end{lemma}

\begin{proof}
Since the group $G$ permutes $N$-orbits and at least one such orbit contains at least two points, choosing a pair of points from~$X$ lying in the same $N$-orbit and a pair of points lying in different $N$-orbits, we obtain a contradiction with $2$-transitivity.

If the center of $G$ is nontrivial, then it acts on $X$ transitive. This shows that the stabilizers in $G$ of points in $X$ are conjugated by elements of the center and so they coincide. This again contradicts to $2$-transitivity.
\end{proof} 

\begin{remark}
Consider the group $G$ of all affine transformations of a vector space $V$ of dimension $\ge 2$ over the field $\ZZ/2\ZZ$. Then $G$ acts on $V$ $3$-transitively, and its normal subgroup $N$ of parallel translations acts on $V$ transitively, but not $2$-transitively. This shows that when passing to a normal subgroup, the transitivity degree can decrease by more than $1$.
\end{remark}

The following result can be found, for example, in ~\cite[Corollary~7.2A]{DM}. For convenience of the reader, we give a proof of this result taken from~\cite[Lemma~5.6]{AZ}.

\begin{theorem} \label{norgr}
Let a group $G$ act on a set $X$ effectively and infinitely transitively, and $N$ be a nontrivial normal subgroup in $G$. Then the action of $N$ on $X$ is also infinitely transitive.
\end{theorem}

\begin{proof}
Consider a set $\alpha=\{x_1,\ldots,x_m\}$ of pairwise distinct elements in $X$. Let $G_{\alpha}$ be the intersection in $G$ of the stabilizers of points from $\alpha$ and set $N_{\alpha}=N\cap G_{\alpha}$. Then $N_{\alpha}$ is a normal subgroup in $G_{\alpha}$. It is sufficient to prove that for any $m$ and any $\alpha$, the action of $N_{\alpha}$ on $X\setminus\{x_1,\ldots,x_m\}$ is transitive. For $m=0$ we obtain $G_{\alpha}=G$, $N_{\alpha}=N$, and the transitivity follows from Lemma~\ref{prtr}.

Suppose that the required property does not hold for some $m$ and some~$\alpha$, and the value $m\ge 1$ is minimal possible. We know that the group $G_{\alpha}$ acts on ${X\setminus\{x_1,\ldots,x_m\}}$ infinitely transitively. It follows from Lemma~\ref{prtr} that if the action of $N_{\alpha}$ on ${X\setminus\{x_1,\ldots,x_m\}}$ is not transitive, then $N_{\alpha}=\{e\}$.

Let $\beta=\{x_1,\ldots,x_{m-1}\}$. From the minimality of $m$ it follows that the group $N_{\beta}$ acts transitively on $X\setminus\{x_1,\ldots,x_{m-1}\}$. The condition $N_{\alpha}=\{e\}$ implies that this action is free, and we can identify the sets $X\setminus\{x_1,\ldots,x_m\}$ and $N_{\beta}\setminus\{e\}$: an element $y$ is identified with an element $h\in N_{\beta}\setminus\{e\}$ such that $y=hx_m$. Under this identification, the action of $G_{\alpha}$ on $N_{\beta}\setminus\{e\}$ is the action by conjugation:
$$
ghx_m=ghg^{-1}gx_m=ghg^{-1}x_m.
$$ 
But a conjugation sends a pair of the form $(h,h^{-1})$ to a pair of the same form. This contradicts to $2$-transitivity of the action of $G_{\alpha}$ on $N_{\beta}\setminus\{e\}$, except in the case when all elements of the group $N_{\beta}$ have order not higher than $2$. In the latter case, the group $N_{\beta}$ can be identified with the additive group $(V,+)$ of some vector space $V$ over the field~$\ZZ/2\ZZ$. Then the group $G_{\alpha}$ acts on $V$ by linear transformations, such an action preserves linear dependencies and, therefore, is not $3$-transitive. The resulting contradiction completes the proof of the theorem.
\end{proof} 

In conclusion, we list several results on infinitely transitive actions of abstract groups.

It is clear that the group of all permutations with a finite support of an infinite set $X$ acts on $X$ infinitely transitively. This group consists of elements of finite order and it is not finitely generated.

It is easy to see that a group $G$ admits an effective infinitely transitive action on a countable set if and only if $G$ can be embedded as a dense subgroup into the infinite symmetric group $S(\NN)$ equipped with the topology of pointwise convergence. In~\cite{Di} it is shown that most (in the topological sense) infinitely transitive finitely generated subgroups in the group $S(\NN)$ are free.

In~\cite{Kit} it is proved that the fundamental group of a closed oriented surface of genus $g>1$ admits an effective infinitely transitive action. A characterization of fundamental groups of three-dimensional compact manifolds admitting such an action is obtained in~\cite{HO}.

The existence of effective infinitely transitive actions of free products $G_1*G_2$ was studied in \cite{GMcC,Gu,Hick,MS}. It turned out that if one of the groups $G_i$ has order at least $3$, then the free product admits an effective infinitely transitive action, whereas the group $(\ZZ/2\ZZ)*(\ZZ/2\ZZ)$ does not allow such an action. The existence of effective infinitely transitive actions for groups acting on trees was studied in~\cite{FLMS,FMS}.

It is shown in \cite{HO} that any group given by at least three generators and one defining relation admits an effective infinitely transitive action.

In~\cite[Question~2]{AZ}, it is asked whether a finitely generated group $G$ acting infinitely transitively on some infinite set $X$ can have intermediate growth. Currently, this issue remains open.

An effective infinitely transitive action of the group of outer automorphisms $\text{Out}{F_n}$ of the free group $F_n$ with $n\ge 4$ is constructed in~\cite{GG}. For $n=3$, the same result is obtained in~~\cite{HO}. It is also shown there that for $n=2$ this group does not allow the required action.

In~\cite[Section~IV.4]{Ch}, an effective infinitely transitive action of a non-elementary hyperbolic group without nontrivial normal finite subgroups is constructed. In~\cite{HO} it is shown that every countable acylindrically hyperbolic group admits an infinitely transitive action with a finite kernel. This result has many important consequences. Also the work~\cite{HO} contains a detailed survey of recent results on multiply transitive actions of countable groups.

\end{document}